\documentclass[english]{amsart}
\usepackage[T1]{fontenc}
\usepackage[latin9]{inputenc}
\usepackage{color}
\usepackage{float}
\usepackage{booktabs}
\usepackage{mathtools}
\usepackage{enumitem}
\usepackage{amstext}
\usepackage{amsthm}
\usepackage{amssymb}
\usepackage{graphicx}
\usepackage{wasysym}

\makeatletter

\providecommand{\tabularnewline}{\\}

\numberwithin{equation}{section}
\numberwithin{figure}{section}
\theoremstyle{plain}
\newtheorem*{thm*}{\protect\theoremname}
\theoremstyle{plain}
\newtheorem{cor}{\protect\corollaryname}
\theoremstyle{definition}
\newtheorem{defn}{\protect\definitionname}
\theoremstyle{plain}
\newtheorem{thm}{\protect\theoremname}
\theoremstyle{plain}
\newtheorem{lem}{\protect\lemmaname}
\theoremstyle{plain}
\newtheorem{prop}{\protect\propositionname}
\theoremstyle{plain}
\newtheorem*{criterion*}{\protect\criterionname}
\theoremstyle{definition}
 \newtheorem{example}{\protect\examplename}
\theoremstyle{remark}
\newtheorem{rem}{\protect\remarkname}
\newlist{casenv}{enumerate}{4}
\setlist[casenv]{leftmargin=*,align=left,widest={iiii}}
\setlist[casenv,1]{label={{\itshape\ \casename} \arabic*.},ref=\arabic*}
\setlist[casenv,2]{label={{\itshape\ \casename} \roman*.},ref=\roman*}
\setlist[casenv,3]{label={{\itshape\ \casename\ \alph*.}},ref=\alph*}
\setlist[casenv,4]{label={{\itshape\ \casename} \arabic*.},ref=\arabic*}

\usepackage{color,graphics}
\usepackage{amsmath}% provides \@mathmeasure

\makeatletter
\DeclareRobustCommand{\prescript}[3]{%
  \@mathmeasure\z@\scriptstyle{#1}%  put the sup in box 0
  \@mathmeasure\tw@\scriptstyle{#2}% put the sub in box 2
  \ifdim\wd\tw@>\wd\z@
    \setbox\z@\hbox to\wd\tw@{\hfil\unhbox\z@}%
  \else
    \setbox\tw@\hbox to\wd\z@{\hfil\unhbox\tw@}%
  \fi
  \mathop{}%
  \mathopen{\vphantom{#3}}^{\box\z@}_{\box\tw@}%
  #3%
}
\makeatother

\usepackage[all]{xy}

\AtBeginDocument{
  
}

\makeatother

\usepackage{babel}
\providecommand{\casename}{Case}
\providecommand{\corollaryname}{Corollary}
\providecommand{\criterionname}{Criterion}
\providecommand{\definitionname}{Definition}
\providecommand{\examplename}{Example}
\providecommand{\lemmaname}{Lemma}
\providecommand{\propositionname}{Proposition}
\providecommand{\remarkname}{Remark}
\providecommand{\theoremname}{Theorem}

\begin{document}
\global\long\def\Der{\textup{Der}}%
\global\long\def\comp{\textup{Comp}}%
\global\long\def\CC{\mathbb{C}}%
\global\long\def\top{\textup{Top}^{\bullet}\left(S\right)}%
\global\long\def\moduli{\mathbb{M}^{\bullet}\left(S\right)}%

\title{The Saito module and the moduli of a germ of curve in $\left(\mathbb{C}^{2},0\right)$.}
\author{Yohann Genzmer}
\begin{abstract}
This article proposes to study the moduli space of a germ of curve
$S$ in the complex plane, that is to say the equisingularity class
of $S$ up to analytical equivalence relation. The first part is devoted
to proving that this last quotient can be endowed with a reasonable
complex structure, yet not canonical. The second part deals with the
computation of its generic dimension in terms of topological invariants
of S. It can be obtained from the study of the valuations of the Saito
module of S, $\textup{Der}\left(\log S\right)$, i.e. the module of
vector fields tangent to $S$.
\end{abstract}

\maketitle

\section*{Introduction.}

The number of moduli of a germ of curve $S$ in $\left(\mathbb{C}^{2},0\right)$
is basically the number of parameters on which depend a topologically
miniversal family for $S.$ It is also the \emph{generic dimension}
of the quotient of the topological class of $S$ up to analytical
equivalence relation, provided that this quotient admits a structure
from which a notion of dimension can be derived. Indeed, this moduli
space defined by the quotient of the topological class of $S$
\[
\left\{ \left.S^{\prime}\right|S^{\prime}\sim_{\textup{top}}S\right\} 
\]
by the following action of $\textup{Diff}\left(\mathbb{C}^{2},0\right)$
\[
\phi\cdot S^{\prime}=\phi\left(S^{\prime}\right),\quad\phi\in\textup{Diff}\left(\mathbb{C}^{2},0\right)
\]
\emph{a priori} has\emph{ }no particular structure beyond being a
set. 

The first determination of such a number of moduli goes back to the
work of Sherwood Ebey in 1965 \cite{MR0176983} who dealt with the
irreducible curves - those having only one irreducible component.
Ebey proved that the moduli space of $S$ carries a complex structure
compatible with a non separated topology and computed the number of
moduli for a particular topological class of curve, namely, that given
by the equation $y^{5}=x^{9}$. In 1973, in \cite{zariski}, Oscar
Zariski proposed various approaches to get the number of moduli for
irreducible curves beyond the case treated previously by Ebey. He
introduced most of the concepts on which future work will be based.
In 1978, Delorme \cite{Delorme1978} studied extensively the case
of an irreducible curve with one Puiseux pair. In 1979, Granger \cite{Granger}
and later, in 1988, Briançon, Granger and Maisonobe \cite{MR922433}
produced an algorithm to compute the number of the moduli of a non
irreducible quasi-homogeneous curve. In 1988, Laudal, Martin and Pfister
in \cite{MR1101844}, improved the work of Delorme and gave an explicit
description of a miniversal family. From 2009, in a series of papers
\cite{MR2509045,MR2781209,MR2996882}, Abramo Hefez and Marcelo Hernandes
greatly improved the previous studies and achieved the analytical
classification of irreducible curves. Their algorithmic approach provided
in particular the number of moduli.

In 2010 and 2011, in \cite{MR2808211,PaulGen}, Emmanuel Paul and
the author described the moduli space of a topologically quasi-homogeneous
curve $S$ as the spaces of leaves of an algebraic foliation defined
on the moduli space of a foliation whose analytic invariant curve
is precisely $S.$ This work initiated an approach based on the theory
of foliations. In 2019, in \cite{YoyoBMS}, the author gave an explicit
formula for the number of moduli for a curve $S$,\emph{ generic}
in its topological class : this formula involves only very elementary
topological invariants of $S$, such as, the topological class of
its desingularization.

The aim of this article is to investigate the full general case, that
is the number of moduli of a germ of curve in the complex plane. We
emphasize that our objective is far from being as ambitious as a complete
analitycal classification, which would require at least some deep
algorithmic procedures, but is rather to obtain a \emph{geometric}
interpretation of these moduli and a procedure to calculate their
number from primitive topological invariants.

This work follows the ideas introduced in \cite{YoyoBMS} and illustrated
in \cite{MR4082253}, which focus on the irreducible case.

Section 1 establishes an extension of the result of Ebey \cite[Theorem 4]{MR0176983}
to the non irreducible curve : it concerns the structure of the moduli
space. As noticed by Ebey himself at the end of its article, its \emph{machinery}
derives from the theory of algebraic groups and depends on the groups
being solvable and \emph{connected}. Therefore, it cannot be directly
carried over to several component curves. Here, we overcome this issue
by considering not only curves but curves enriched with a \emph{marking
}which allows us to recover the necessary connexity. As Ebey, we use
an adapted complete topological invariant - the semi ring of values
- introduced by R. Waldi \cite{Waldi} and some of its properties
identified by M. Hernandes and E. de Carvalho in \cite{Hernandes.semiring}.
Finally, we obtain the following result
\begin{thm*}
The marked moduli space $\moduli$ of a germ of curve $S$ in $\mathbb{C}^{2}$,
that is its marked topological class up to analytical marked equivalence
relation, can be identified with the quotient of a complex constructible
set by an action of a connected solvable algebraic group. In particular,
it is endowed with a non separated complex structure.
\end{thm*}
Notice that passing from the moduli space to the \emph{marking} moduli
space has no effect on the generic dimension.

Section 1 can be read independently from the rest of the article.

Section 2 and 3 aim to develop the study of the module $\textup{Der}\left(\log S\right)$
of vector fields tangent to $S,$ on which depends the computation
of the number of moduli of $S.$ The starting point is a remark of
K. Saito in \cite{MR586450}, that, highlighted the freeness of this
module - which is specific to the curves embedded in the complex plane.
An immediate consequence of the work of Saito is that, the smallest
valuation of the vector fields in $\textup{Der}\left(\log S\right)$
cannot be too big compared to the valuation of $S,$ namely, the following
upper bound holds 
\[
\min_{X\in\textup{Der}\left(\log S\right)}\nu\left(X\right)\leq\frac{\nu\left(S\right)}{2}.
\]
Our purpose is to prove that, \emph{generically, }this bound is essentially
reached. In section 2, the existence of a \emph{flat} basis of $\textup{Der}\left(\log S\right)$
is shown in the generic situation, that is, a basis admitting an analytic
extension as a basis for the modules $\textup{Der}\left(\log C\right)$
where $C$ are in a \emph{neighborhood} of $S$ in $\moduli$. As
a consequence, using the theory of infinitesimal deformations of foliations
of X. G\'{o}mez-Mont \cite{Gomez}, we obtain the following theorem
\begin{thm*}
For $S$ generic in its moduli space $\moduli$, one has 
\[
\min_{X\in\textup{Der}\left(\log S\right)}\nu\left(X\right)\geq\left\{ \begin{array}{ll}
\left\lfloor \frac{\nu\left(S\right)}{2}\right\rfloor  & \textup{ if }S\textup{ is not of radial type}\\
\\
\left\lceil \frac{\nu\left(S\right)}{2}\right\rceil -1 & \textup{else}
\end{array}\right..
\]
\end{thm*}
The definition of $S$ being of \emph{radial type }will be given in
the article. Note that if is $S$ is not generic in its moduli space,
the above lower bound is false, as it will be illustrated by some
examples in the article. Moreover,\emph{ }in section 3, we proceed
with the precise description of the various possibilities for the
flat basis of $\textup{Der}\left(\log S\right)$.

Finally, Section 4 illustrates our approach for the computation of
the generic dimension of $\moduli$. As a consequence of section 2
and 3, we recover the classical dimension of the moduli space of the
singularity $x^{n}+y^{n}=0$ with $n\geq1.$
\begin{cor}[\cite{Granger}]
The generic dimension of $\moduli$ where 
\[
S=\left\{ x^{n}+y^{n}=0\right\} 
\]
 is equal to 
\[
\left\{ \begin{array}{ll}
\frac{\left(n-2\right)^{2}}{4} & \textup{ if \ensuremath{n} is even}\\
\frac{\left(n-1\right)\left(n-3\right)}{4} & \text{\textup{ if \ensuremath{n} is odd.}}
\end{array}\right.
\]
\end{cor}
In an upcoming article, we will build an algorithm based upon the
results presented here, that computes the generic dimension of the
moduli space for more general curves. We implemented, among other
procedures, this algorithm on Sage 9.{*}. See the routine \emph{Courbes.Planes}
following the link
\begin{quote}
https://perso.math.univ-toulouse.fr/genzmer/
\end{quote}

\section{Moduli space of marked curve.}

Throughout this article, $S$ stands for a germ of singular curve
in the complex plane $\left(\CC^{2},0\right)$. In particular, its
algebraic valuation is at least $2.$ From now on, we fix a decomposition
of $S$ in irreducible components 
\[
S=S^{1}\cup\cdots\cup S^{r}
\]
where $r$ is the number of irreducible components. Here and subsequently,
$\comp\left(S\right)$ stands for the set of the irreducible components
of $S$.

Let $C$ be a germ of curve topologically equivalent to $S$ by a
germ of homeomorphism of the ambient space $\left(\mathbb{C}^{2},0\right)$
denoted by $h$ and such that 
\[
h\left(S\right)=C.
\]

The application $h$ induces a bijective map 
\[
\sigma_{h}:\comp\left(S\right)\to\comp\left(C\right)
\]
Two such homeomorphisms $h$ and $h^{\prime}$ are said to be \emph{equivalent}
if and only if
\begin{equation}
\sigma_{h}=\sigma_{h^{\prime}}.\label{relation.equivalence.homeo}
\end{equation}

\begin{defn}
A curve \emph{marked by $S$} is a couple $\left(C,\overline{h}\right)$
where $C$ is curve topologically equivalent to $S$ and $\overline{h}$
a class of homeomorphism between $C$ and $S$ for the equivalence
relation defined above. We will denote by $\textup{Top}^{\bullet}\left(S\right)$
the set of curves marked by $S.$
\end{defn}
The group $\textup{Diff}\left(\mathbb{C}^{2},0\right)$ of germs of
automorphisms of the ambient space $\left(\CC^{2},0\right)$ acts
on the set $\textup{Top}^{\bullet}\left(S\right)$ by
\[
\phi\cdot\left(C,\overline{h}\right)=\left(\phi\left(C\right),\overline{\phi\circ h}\right).
\]
In what follows, the quotient of $\textup{Top}^{\bullet}\left(S\right)$
by $\textup{Diff}\left(\mathbb{C}^{2},0\right)$ will be denoted by
\[
\mathbb{M}^{\bullet}\left(S\right)
\]
and will be refered to as \emph{the marked moduli space of $S.$ }Although
$\moduli$ cannot be endowed with a complex structure by some general
statements about group actions, the result below provides such a structure.
Indeed, generalizing a result of Ebey \cite{MR0176983}, we obtain
the
\begin{thm}
\label{theorem.structure.moduli}The quotient $\mathbb{M}^{\bullet}\left(S\right)$
can be identified with the quotient of a complex constructible set
by an action of a connected solvable algebraic group.
\end{thm}
This result still holds if we drop the assumption of $S$ being a
\emph{plane }curve, once we replace the topological equivalence by
the equisingularity which corresponds to the equality of the semirings
of valuations as definded in \emph{\cite{Hernandes.semiring}}. Since
the general proof consists at most in increasing the complexity of
the notations, we state Theorem \ref{theorem.structure.moduli} and
prove it only for a curve embedded in the complex plane. We follow
Theorem 5 in \cite{MR0176983} observing that a connected solvable
algebraic action on a complex constructible set admits a complete
transversal, that is a constructible subset in correspondance one
to one with the orbits of the action. Thus, from Theorem \ref{theorem.structure.moduli},
$\moduli$ inherits of the complex structure of this transversal.
Its compatible topology is just the quotient topology : in most case,
it is not separated (see for instance \cite{MR2509045,MR2781209}).

The goal of the current section is to prove Theorem \ref{theorem.structure.moduli}.

\subsection{The ring of of functions of $\left(C,\overline{h}\right)$}

Let $\left(C,\overline{h}\right)$ be in $\top$ and 
\[
\gamma_{C}=\left\{ \gamma_{c}:t\in\left(\CC,0\right)\to\left(\CC^{2},0\right)\right\} _{c\in\textup{Comp}\left(C\right)}
\]
be any system of parametrizations of the irreducible components of
$C.$ We denote by $C_{i}$ the component of $C$ defined by the marking
$\overline{h}$
\[
C_{i}=\sigma_{h}\left(S^{i}\right).
\]

From the marking $\overline{h}$ of $C$, we derived a morphism of
rings defined by 
\[
\left\{ \begin{array}{ccc}
\mathbb{C}\left[\left[x,y\right]\right] & \to & \left(\mathbb{C}\left[\left[t\right]\right]\right)^{r}\\
u & \mapsto & \left(\gamma_{C_{i}}^{\star}u\right)_{i=1,\ldots,r}
\end{array}\right..
\]
which factorizes in an monomorphism 
\begin{equation}
\mathfrak{E}_{\left(C,\overline{h}\right)}:\mathcal{\hat{O}}_{C}=\frac{\mathbb{C}\left[\left[x,y\right]\right]}{\left(f\right)}\hookrightarrow\left(\mathbb{C}\left[\left[t\right]\right]\right)^{r}\label{epimorpshims}
\end{equation}
where $f$ is any reduced equation of $C$ and $\mathcal{\hat{O}}_{C}$
is the completion of $\mathcal{O}_{C}=\frac{\CC\left\{ x,y\right\} }{\left(f\right)}.$

The following result is classic - see \cite{MR0176983} for the irreducible
case. 
\begin{lem}
\label{LemmeIsoAnneau}Let $\left(C,\overline{h}\right)$ and $\left(C^{\prime},\overline{h^{\prime}}\right)$
be two marked curves in $\top.$ The following propertie are equivalent
\begin{enumerate}
\item The curves $\left(C,\overline{h}\right)$ and $\left(C^{\prime},\overline{h^{\prime}}\right)$
are analytically equivalent by a conjugacy preserving the markings.
\item The images of the monomorphisms (\ref{epimorpshims}) associated to
both curves are conjugated by a diagonal formal automorphism of $\left(\mathbb{C}\left[\left[t\right]\right]\right)^{r}$.
\end{enumerate}
\end{lem}

\subsection{The tropical semiring of values of $\left(C,\overline{h}\right)$}

Following \cite{Hernandes.semiring}, we consider $\Gamma_{C}$ the
set defined by 
\[
\Gamma_{C}=\left\{ \left.\nu\left(G\right)\right|G\in\mathcal{O}_{C}\right\} \subset\mathbb{\left(\overline{N}\right)}^{r}
\]
where $\overline{\mathbb{N}}=\mathbb{N}\cup\left\{ \infty\right\} $.
The valuation $\nu$ is defined by 
\[
\nu\left(G\right)=\left(\nu_{0}\left(\gamma_{C_{i}}^{\star}G\right)\right)_{i=1,\ldots,r}
\]
where $\nu_{0}$ is the standard valuation $\mathbb{C}\left\{ t\right\} .$
Notice that this set depends not only on the curve $C$ but also on
its marking.

The set $\Gamma_{C}$ inherits of a semiring structure defined by
\[
\alpha\oplus\beta=\left(\min\left\{ \alpha_{i},\beta_{i}\right\} \right)_{i=1\ldots r}\qquad\alpha\astrosun\beta=\left(\alpha_{i}+\beta_{i}\right)_{i=1\ldots r}
\]
where we set $k+\infty=\infty.$ $\Gamma_{C}$ is also partially ordered
by the product order $\leq$. The\emph{ }quadruplet\emph{ $\left(\Gamma_{C},\oplus,\astrosun,\leq\right)$
}is\emph{ }called\emph{ the tropical semiring of values of} $\left(C,\overline{h}\right)$.
\begin{defn}
\label{def:irreducible.absolute}A element $\alpha\in\Gamma_{C}$
is said \emph{irreducible} if and only if 
\[
\left(\alpha=a+b\textup{ with }a,\ b\in\Gamma_{C}\right)\Longrightarrow\alpha=a\textup{ or }\alpha=b.
\]
It is said to be \emph{absolute} if for any non empty proper subset
$J$ of the set 
\begin{equation}
\mathcal{I}_{\alpha}=\left\{ \left.i\in\left\{ 1,\ldots r\right\} \right|\alpha_{i}\neq\infty\right\} ,\label{coordonnees.a.infini}
\end{equation}
the following set 
\begin{equation}
F_{J}\left(\alpha\right)=\left\{ \left.a\in\Gamma_{C}\right|\forall i\in\mathcal{I}_{\alpha}\setminus J,a_{i}>\alpha_{i}\textup{ and }\forall i\notin\mathcal{I}_{\alpha}\setminus J,a_{i}=\alpha_{i}\right\} \label{ensemble.non.nul.audessus}
\end{equation}
is empty.
\end{defn}
The following result gathers some known properties of the semiring
of values.
\begin{thm}[\cite{MR909850,Hernandes.semiring,Waldi}]
\label{basic-facts}Two germ of plane curves are topologically equivalent
if and only if they share the same semiring of values \cite{Waldi}.
More precisely, if $C_{1}\cup C_{2}\cup\cdots\cup C_{r}$ and $C_{1}^{\prime}\cup C_{2}^{\prime}\cup\cdots\cup C_{r}^{\prime}$
are two curves with same semiring, then there exists an homeomorphism
$\phi$ of the ambient space $\left(\mathbb{C}^{2},0\right)$ such
that for any $i$ 
\[
\phi\left(C_{i}\right)=C_{i}^{\prime}.
\]
Moreover,
\begin{enumerate}
\item $\Gamma_{C}$ has a conductor, i.e, there exists a minimal $\sigma\in\Gamma_{C}$
such that $\sigma+\overline{\mathbb{N}}^{r}\subset\Gamma_{C}$ \cite{MR909850}.
\item The set $g$ of irreducible absolute points of $\Gamma_{C}$ is finite
and minimaly generates $\Gamma_{C}$ as semiring \cite{Hernandes.semiring}.
\item Any family $G$ of $\mathcal{O}_{C}$ such that $\nu\left(G\right)=g$
is a minimal standard basis of $\mathcal{O}_{C}$ as defined in \cite{Hernandes.semiring}.
\end{enumerate}
\end{thm}
In particular, for any element $C\in\top$, one has 
\[
\Gamma_{C}=\Gamma_{S}.
\]

For now on, we will denote the mutual semiring for curves in $\top$
simply by $\Gamma.$

\subsection{Trunctation and conductor.}

The following lemma allows us to truncate elements in the ring $\mathcal{O}_{C}$
(resp. $\hat{\mathcal{O}}_{C}$ ).
\begin{lem}
\label{truncation}Suppose that $G=\left(\sum_{k=0}^{\infty}a_{lk}t^{k}\right)_{l=1,\ldots,r}$
is an element of $\mathcal{O}_{C}$ ( resp. of its completion $\hat{\mathcal{O}}_{C}$).
Then , for any $p=\left(p_{1},\ldots,p_{r}\right)\in\mathbb{N}^{r}$
with $p_{l}\geq\sigma_{l}-1$ for $l=1,\ldots,r$, one has 
\[
\left(\sum_{k=0}^{p_{l}}a_{lk}t^{k}\right)_{l=1,\ldots,r}\in\mathcal{O}_{C},\ \textup{(resp. \ensuremath{\hat{\mathcal{O}}_{C}}}\textup{)}
\]
\end{lem}
\begin{proof}
By definition of $\sigma,$ for any $l=1,\ldots,r$ and for any $k\geq p_{l}+1\geq\sigma_{l}$,
the $r-$uple 
\[
\left(\infty,\cdots,\infty,\underbrace{k}_{l^{th}},\infty,\cdots,\infty\right)
\]
belongs to $\Gamma.$ Thus, an inductive argument on the rank $k\geq p_{l}+1$
shows that there exists a formal series $\hat{F_{l}}\in\mathbb{C}\left[\left[x,y\right]\right]$
such that 
\[
\gamma^{\star}\hat{F_{l}}=\left(0,\cdots,0,\sum_{k=p_{l}+1}^{\infty}a_{lk}t^{k},0,\cdots,0\right).
\]
 Now, following \cite[Theorem 1, p.493]{MM}, if $G$ is convergent,
so is $\hat{F}_{l}$ and, in any case, evaluating 
\[
G-\gamma^{\star}\left(\sum_{l=1}^{r}\hat{F_{l}}\right)
\]
yields the lemma.
\end{proof}

\subsection{$\Gamma-$reduction.}

The notion of\emph{ $\Gamma-$reduction} will allow us to construct
normal forms for systems of generators of $\hat{\mathcal{O}}_{C}$.

Let $\underline{P}=\left(P_{i}\right)_{i=1,\ldots,r}$ be a family
of $r$ finite subsets of $\overline{\mathbb{N}}$ such that for any
$i,$ $\infty\in P_{i}.$
\begin{defn}
The family $\underline{P}$ is \emph{said to be $\Gamma$-reduced}
if and only if 
\[
\Gamma\cap\prod_{i=1,\ldots,r}P_{i}=\left\{ \underline{\infty}\right\} 
\]
where $\underline{\infty}=\left(\infty,\infty,\cdots,\infty\right)$
\end{defn}
A\emph{ $\Gamma-$reduction of $\underline{P}$ }is an elementary
transformation of $\underline{P}$ of the following form : suppose
that there exists $\underline{n}=\left(n_{1},\cdots,n_{r}\right)\neq\underline{\infty}$
such that 
\[
\underline{n}\in\Gamma\cap\prod_{i=1,\ldots,r}P_{i}.
\]
Consider an integer $i$ such that $n_{i}\neq\infty.$ Then the family
$\underline{P}^{\left(1\right)}=\left(P_{i}^{\left(1\right)}\right)_{i=1,\ldots,r}$
defined by 
\begin{align*}
\left\{ \begin{array}{clc}
P_{j}^{\left(1\right)}= & P_{j} & \textup{for }j\ne i\\
P_{i}^{\left(1\right)}= & P_{i}\setminus\left\{ n_{i}\right\} 
\end{array}\right.
\end{align*}
is called a $\Gamma-$reduction of $\underline{P}.$ To keep track
of a $\Gamma-$reduction, we denote it by 
\[
\underline{P}=\underline{P}^{\left(0\right)}\xrightarrow{\underline{n},i}\underline{P}^{\left(1\right)}.
\]
The following lemma is obvious
\begin{lem}
For any $\underline{P}$, there exists a finite sequence of $\Gamma-$reductions
\[
\underline{P}=\underline{P}^{\left(0\right)}\xrightarrow{\underline{n}_{0},i_{0}}\underline{P}^{\left(1\right)}\xrightarrow{\underline{n}_{1},i_{1}}\cdots\xrightarrow{\underline{n}_{q-1},i_{q-1}}\underline{P}^{\left(q\right)}.
\]
such that $\underline{P}^{\left(q\right)}$ is $\Gamma-$reduced.
\end{lem}
Notice that this sequence is not unique.

\subsection{Parametrization of the set $\protect\top$.}

Let $g=\left\{ g^{1},\cdots,g^{q}\right\} $ be the set of irreducible
absolute points of $\Gamma$ and $G=\left\{ G^{1},\cdots,G^{q}\right\} \subset\mathcal{O}_{C}$
such that for all $i,$
\[
\nu\left(G^{i}\right)=g^{i}.
\]

\begin{lem}
Among the family $G$ and in the identification
\[
\mathcal{O}_{C}=\frac{\mathbb{C}\left[\left[x,y\right]\right]}{\left(f\right)},
\]
there are exactly two components $G^{i}$ whose linear parts are independant.
\end{lem}
\begin{proof}
Assume that $C$ contains an irreducible singular component, say $C_{1}$,
and consider some coordinates $\left(x,y\right)$ such that it is
parametrized by
\[
t\to\left(t^{n},t^{m}+\cdots\right),\quad n\nmid m
\]
Evaluating the valuation of the coordinate functions $x$ and $y$,
we obtain that $\Gamma$ contains two elements of the form 
\begin{equation}
\nu\left(x\right)=\left(n,\cdots\right)\in\Gamma\textup{ and }\nu\left(y\right)=\left(m,\cdots\right)\in\Gamma.\label{bothpoints}
\end{equation}
If the linear parts of the functions $G^{i}$ are dependant two by
two, then the set of valuations of the complete ring generated by
the family $G$ can contains either $\left(n,\cdots\right)$ or $\left(m,\cdots\right)$
or none of them, but certainly not both. However, according to Theorem
\ref{basic-facts}, the complete ring generated by $G$ is the whole
ring $\hat{\mathcal{O}}_{C}$, which contradicts (\ref{bothpoints}).
If $C$ contains two smooth components, transversal or not, a contradiction
can be obtained in much the same way by considering coordinates in
which these components are written 
\[
t\to\left(\left(t,0\right),\left(0,t\right),\cdots\right)\textup{ or }t\to\left(\left(t,0\right),\left(t,t^{n}\right),\cdots\right),\ n\geq2.
\]
Finally, the existence of a third function $G^{i}$ with a non trivial
linear part would contradict the minimality of the family $g$.
\end{proof}
Changing the numbering of the elements in $g,$ we may assume that
the two elements identified by the above lemma are $G^{1}$ and $G^{2}$
with $g^{1}<g^{2}$ for the lexicographic order. Let us denote $G^{i}$,
$i=1,2$ by
\begin{equation}
G^{i}=\left(\sum_{k=g_{l}^{i}}^{\infty}a_{lk}^{i}t^{k}\right)_{l=1,\ldots,r}\label{truncature-1}
\end{equation}
Notice that in the above expression, $g_{l}^{i}$ may be equal to
$\infty$ and the corresponding component $\left(G^{i}\right)_{l}$
be equal to $0$ . However, one has the following
\begin{lem}
Assume that $C$ is not the union of two smooth curves. If $g_{l}^{i}\neq\infty,$
then $g_{l}^{i}\leq\sigma_{l}-1.$
\end{lem}
\begin{proof}
The proof is by contradiction. Suppose that for some $l,$ $g_{l}^{i}\neq\infty$
and $g_{l}^{i}\geq\sigma_{l}.$ Applying Lemma \ref{truncation} to
$G^{i}$ with 
\[
\left(p_{i}\right)_{i=1,\ldots,r}=\left(\infty,\cdots,\infty,\sigma_{l},\infty,\cdots,\infty\right)
\]
yields an element $\overline{g}\in\Gamma$ such that $\overline{g}_{l}=\infty$
and $\overline{g}_{k}=g_{k}^{i}$ for $k\neq l.$ Consider the proper
subset of $\mathcal{I}_{g^{i}}$ defined by 
\[
J=\mathcal{I}_{g^{i}}\setminus\left\{ l\right\} ,
\]
and suppose it is non empty. Definition \ref{def:irreducible.absolute}
of absolute point ensures that $F_{J}\left(g^{i}\right)$ is empty.
However, by construction, $\overline{g}$ belongs to $F_{J}\left(g^{i}\right)$
which is a contradiction. Thus, $J$ is empty and $\mathcal{I}_{g^{i}}=\left\{ l\right\} $.
Therefore, $g^{i}$ is written 
\[
g^{i}=\left(\infty,\cdots,\infty,g_{l}^{i},\infty,\cdots,\infty\right).
\]
\begin{itemize}
\item If $r\geq3,$ we are lead to a contradiction noticing that $G^{i}$
would be a function with non trivial linear part vanishing along two
distinct components of $C.$
\item Assume $r=2.$ Since $G^{i}$ is a regular function and $g^{i}=\left(\infty,g_{2}^{i}\right)$
or $\left(g_{1}^{i},\infty\right)$, one of the component of $C$,
say $C_{1},$ is smooth. One can choose some coordinates $\left(x,y\right)$
such that
\begin{align*}
C_{1} & =\left\{ \alpha y+\beta x=0\right\} ,\ \alpha,\beta\in\mathbb{C}\\
C_{2} & =\left\{ y^{p}+x^{q}+\cdots=0\right\} 
\end{align*}
with $p<q$. The hypothesis of the lemma ensures that the case $p=1$
is excluded. According to \cite{Hernandes.semiring}, the conductor
$\sigma$ of $C$ is written 
\[
\sigma=\left(0,c_{2}\right)+\left\{ \begin{array}{cc}
\left(p,p\right) & \textup{ if }\beta\neq0\\
\left(q,q\right) & \textup{ if }\beta=0
\end{array}\right..
\]
where $c_{2}\geq1$ is the conductor of the component $C_{2}.$ By
construction, the function $G^{1}$ is equal to $\alpha y+\beta x.$
Therefore, 
\[
g^{1}=\left(\infty,\left\{ \begin{array}{cc}
p & \textup{ if }\beta\neq0\\
q & \textup{ if }\beta=0
\end{array}\right.\right),
\]
thus $g_{2}^{1}<\sigma_{2}.$
\end{itemize}
\end{proof}
If $C$ is a union of two smooth curves then one has
\[
\sigma=\left(n,n\right),\ g^{1}=\left(\infty,n\right)\textup{ and }g^{2}=\left(n,\infty\right)
\]
where $n$ is the order of tangency between $C_{1}$ and $C_{2}.$
Actually, $C$ is analytically equivalent to the curve
\[
y\left(y+x^{n}\right)=0.
\]
Thus the moduli space of $C$ reduces to a point and the problem of
the analytic classification is trivial. For now on, we will assume
that $C$ is not a union of two smooth curves.

Lemma \ref{truncation} yields truncations of $G^{i},$ $i=1,2$ that
we keep on denoting it by 
\begin{equation}
G^{i}=\left(\sum_{k=g_{l}^{i}}^{\max\left(\sigma_{l}-1,g_{l}^{i}\right)}a_{lk}^{i}t^{k}\right)_{l=1,\ldots,r}\label{truncature}
\end{equation}
Notice that some components of (\ref{truncature}) - but not all -
may vanish.

We are going to normalize the expressions of $G^{i}$ in order to
make it unique and depending only on the marked curve $\left(C,\overline{h}\right).$
The first normalization consists in the following: for $i=1,\ 2$
let us consider the smallest $l_{i}$ such that $g_{l_{i}}^{i}\neq\infty,$
we impose that 
\[
a_{l_{i}g_{l_{i}}^{i}}^{i}=1.
\]
To go further in the normalization, we will use $\Gamma-$reductions.
For $i=1,2\ $ let us consider $\underline{P}^{i}=\left(P_{1}^{i},\ldots,P_{r}^{i}\right)$
defined by 
\[
\left\{ \begin{array}{lc}
P_{l}^{i}=\left[g_{l}^{i},\max\left(\sigma_{l}-1,g_{l}^{i}\right)\right]\cap\mathbb{N}\cup\left\{ \infty\right\}  & \textup{ if }l\neq l_{i}\\
P_{l_{i}}^{i}=\left[g_{l_{i}}^{i}+1,\max\left(\sigma_{l_{i}}-1,g_{l_{i}}^{i}\right)\right]\cap\mathbb{N}\cup\left\{ \infty\right\} 
\end{array}\right.
\]
Notice that if $g_{l}^{i}=\infty$ then $P_{l}^{i}=\left\{ \infty\right\} .$
In the same way, if $g_{l_{i}}^{i}=\sigma_{l_{i}}-1$ then $P_{l_{i}}^{i}=\left\{ \infty\right\} .$

For any $\underline{n}\in\overline{\mathbb{N}}^{r},$ we denote by
$\textup{Init}^{i}\left(\underline{n}\right)$ the integer defined
by 
\[
\min\left\{ \left.k\right|\left(\underline{n}\right)_{k}\neq\infty\textup{ and }\left(\underline{n}\right)_{k}\neq g_{k}^{i}\right\} .
\]
If $\underline{n}\neq\underline{\infty}\in\Gamma\cap\prod_{l=1,\ldots,r}P_{l}^{i}$
then $\textup{Init}^{i}\left(\underline{n}\right)$ is well defined
since the set of which it is the minimum is non-empty : indeed, if
for any $l$, one has $\left(\underline{n}\right)_{l}=g_{l}^{i}$
or $\infty$. In particular, $\left(\underline{n}\right)_{l_{i}}=\infty.$
Moreover, $\underline{n}$ belongs to $F_{J}\left(g^{i}\right)$ where
$J$ is defined by 
\[
J=\left\{ \left.k\right|\left(\underline{n}\right)_{k}\neq\infty\right\} .
\]
The set $J$ is non-empty since $\underline{n}\neq\underline{\infty}$
and is proper since $l_{i}\notin J$. That is impossible because by
definition of absolute point, $F_{J}\left(g^{i}\right)$ is empty.

We choose a sequence of $\Gamma-$reductions of $\underline{P}^{i}$
\[
\underline{P}^{i}=\underline{P}^{i,\left(0\right)}\xrightarrow{\underline{n}_{0},k_{0}}\underline{P}^{i,\left(1\right)}\xrightarrow{\underline{n}_{1},k_{1}}\cdots\underline{P}^{i,\left(q_{i}-1\right)}\xrightarrow{\underline{n}_{q_{i}-1},k_{q_{i}-1}}\underline{P}^{i,\left(q_{i}\right)}
\]

such that for any $t\in\left\{ 0,\cdots,q_{i}-1\right\} ,$
\begin{enumerate}
\item[$\left(P_{1}\right)$]  $\textup{Init}^{i}\left(\underline{n}_{t}\right)={\displaystyle \min}\left\{ \textup{Init}^{i}\left(\underline{n}\right)\left|\underline{n}\neq\underline{\infty}\in\Gamma\cap\prod_{l=1,\ldots,r}P_{l}^{i,\left(t\right)}\right.\right\} $
and
\item[$\left(P_{2}\right)$] among the $\underline{n}_{t}$'s that satisfy the previous equality,
we choose the one for which the integer
\[
\left(\underline{n}_{t}\right)_{\textup{Init}^{i}\left(\underline{n}_{t}\right)}
\]
is the smallest possible.
\end{enumerate}
Let us show how the $\Gamma-$reduction 
\[
\underline{P}^{i,\left(t\right)}\xrightarrow{\underline{n}_{t},k_{t}}\underline{P}^{i,\left(t+1\right)}
\]
allows us to normalize $G^{i}.$ The $r-$uple $\underline{n}_{t}$
being an element of $\Gamma$, by definition, there exists a sum of
the form 
\[
W^{i,\left(t\right)}=\sum_{\beta\in\mathbb{N}^{2}}w_{\beta}^{i,\left(t\right)}\left(G^{1}\right)^{\beta_{1}}\left(G^{2}\right)^{\beta_{2}}
\]
such that $\nu\left(W^{i,\left(t\right)}\right)=\underline{n}_{t}$
and the coefficient of $t^{\left(\underline{n}_{t}\right)_{k_{t}}}$
in the $k_{t}^{\textup{th}}$ component is equal to $1.$ The difference
\begin{equation}
G^{i}-a_{k_{t}\left(\underline{n}_{t}\right)_{k_{t}}}^{i}W^{i,\left(t\right)}\label{eq:transformation}
\end{equation}
belongs to $\hat{\mathcal{O}}_{C}$ and the coefficient of $t^{\left(\underline{n}_{t}\right)_{k_{t}}}$
in the $k_{t}^{\textup{th}}$ component vanishes. By construction,
after a $\Gamma-$reduction, the new couple of functions defined by
(\ref{eq:transformation}) still generates $\hat{\mathcal{O}}_{C}$.
Doing the whole process of $\Gamma-$reductions for both $G^{i},~i=1,2$\emph{
}and a final truncation at $\sigma$, we obtain\emph{ a normalized
family of generators that we denote $\left(\mathfrak{N}^{i}\left(G^{i}\right)\right)_{i=1,2}$}.
By construction, following the properties $\left(P_{1}\right)$ and
$\left(P_{2}\right)$, a normalized family of generators $\left(\mathfrak{N}^{i}\left(G^{i}\right)\right)_{i=1,2}$
is written
\begin{equation}
\mathfrak{N}^{i}\left(G^{i}\right)=\left(\sum_{k\in P_{1}^{i,\left(q_{i}\right)}}a_{1k}^{i}t^{k},\cdots,t^{g_{l_{i}}^{i}}+\sum_{k\in P_{l_{i}}^{i,\left(q_{i}\right)}}a_{l_{i}k}^{i}t^{k},\cdots,\sum_{k\in P_{r}^{i,\left(q_{i}\right)}}a_{rk}^{i}t^{k}\right).\label{generators.normalized}
\end{equation}
The main characteristic of this normalized basis is that its parameters
are unique: indeed, $G$ and $G^{\prime}$ being two couples of normalized
generators as in (\ref{generators.normalized}), we consider the valuation
\[
\gamma=\nu\left(G^{i}-\left(G^{\prime}\right)^{i}\right).
\]
By definition, $\gamma$ is an element of $\Gamma$. By construction
of the normalized family, it is also an element of $\prod_{l=1,\ldots,r}P_{l}^{i,\left(q_{i}\right)}$.
Since $\underline{P}^{i,\left(q_{i}\right)}$ is $\Gamma-$reduced,
$\gamma$ is equal to $\underline{\infty}$ and $G^{i}$ and $\left(G^{\prime}\right)^{i}$
are equal. Therefore the normalized basis is unique and we can consider
the following \emph{well defined} map
\[
\mathbb{M}_{S}:\left\{ \begin{array}{rcl}
\top & \longrightarrow & \underset{{\scriptscriptstyle l,i}}{{\displaystyle \prod}}\mathbb{C}^{P_{l}^{i,\left(q_{i}\right)}}\\
\left(C,\overline{h}\right) & \longmapsto & \left(a_{lk}^{i}\right)
\end{array}\right.
\]
that associates to a marked curve in $\top$, the ordered coefficients
of a normalized family of generators of $\mathcal{O}_{C}.$

\subsection{$\protect\top$ as a constructible set.}

\global\long\def\GG{\prescript{\mathfrak{n}}{}{G}^{1}}%
\global\long\def\GGG{\prescript{\mathfrak{n}}{}{G}^{2}}%

In this section, we are going to prove the
\begin{prop}
The image of $\mathbb{M}_{S}$ is a constructible analytic set, i.e,
a finite union of finite intersections of algebraic subsets and complements
of algebraic subsets of the affine set $\underset{l,i}{{\displaystyle \prod}}\mathbb{C}^{P_{l}^{i,\left(q_{i}\right)}}$.
\end{prop}
\begin{proof}
Consider an element of $\underset{l,i}{{\displaystyle \prod}}\mathbb{C}^{P_{l}^{i,\left(q_{i}\right)}}$
and the associated couple $\left(G^{1},G^{2}\right)$ as in (\ref{generators.normalized}).
The complete ring generated by $G$ is the completion of the ring
of a plane curve $C$ with $r$ components $C^{1},\ldots,C^{r}$ given
by the coordinates of $G$. Fix some $i$ in $\left\{ 1,\ldots,r\right\} $.
We begin by proving that the condition
\[
g^{i}\in\Gamma_{C}
\]
is a constructible condition. Choose any reduced equation $h_{i}\left(x,y\right)$
of the curve 
\[
\bigcup_{j\notin\mathcal{I}_{g^{i}}}C^{j}.
\]
If $\mathcal{I}_{g^{i}}$ is empty, choose simply $h_{i}=1.$ Consider
$\mathcal{N}$ the finite set of couples of integers $\left(u,v\right)\in\mathbb{N}^{2}$
such that 
\[
\nu\left(h_{i}\left(G^{1},G^{2}\right)\left(G^{1}\right)^{u}\left(G^{2}\right)^{v}\right)\not\geq\sigma.
\]
and the set of expressions of the form 
\begin{equation}
h_{i}\left(G^{1},G^{2}\right)\times\sum_{\left(u,v\right)\in\mathcal{N}}\beta_{uv}\left(G^{1}\right)^{u}\left(G^{2}\right)^{v},\qquad\beta_{uv}\in\mathbb{C}\label{sumrelation}
\end{equation}
where the $\beta_{uv}$'s are coefficients. It follows that $g^{i}\in\Gamma_{C}$
is equivalent to the existence of a family $\left\{ \beta_{uv}\right\} _{uv}$
so that the expression (\ref{sumrelation}) has a valuation equal
to $g^{i}$. Let 
\[
L_{l,k}^{i}
\]
be the coefficient of $t^{k}$ in the $l^{th}$ component of (\ref{sumrelation}).
The functions $L_{l,k}^{i}$ are linear forms in the variables $\beta_{uv}$
whose coefficients are algebraic expressions in the coefficients of
the generators $G^{i}$. The condition $g^{i}\in\Gamma_{C}$ is equivalent
to require that for each $l=1,\ldots,r,$ the linear form $L_{l,g_{l}^{i}}^{i}$
is linearly independent of the linear forms $L_{l,k}^{i}$ for $l=1,\ldots,r$
and $k<g_{l}^{i}$. The latter condition is constructible one in the
coefficients of the generators $G^{i}$ since it can be expressed
using the ranks of the minors of the matrix of these linear forms.
It follows that $g^{i}\subset\Gamma_{C}$ and thus 
\[
\Gamma\subset\Gamma_{C}
\]
is a constructible condition. We can now proceed analogously to prove
that $\Gamma=\Gamma_{C}$ is also a constructible condition : indeed,
according to \cite{Hernandes.semiring}, providing that $\Gamma\subset\Gamma_{C}$,
the equality $\Gamma=\Gamma_{C}$ is equivalent to the equality 
\[
\Gamma\cap\prod_{i=1}^{r}\left[0,\sigma_{l}\right]=\Gamma_{C}\cap\prod_{i=1}^{r}\left[0,\sigma_{l}\right]
\]
which induces a finite number of conditions, that can be proven to
be constructible with similar arguments.
\end{proof}

\subsection{Action on $\protect\top$.}

The group $\left(\textup{Diff}\left(\CC,0\right)\right)^{r}$ acts
on the image of $\mathfrak{\mathbb{M}}_{S}$ the following way : given
a point $A$ in the image, consider its corresponding couple of generators
$\left(G^{1},G^{2}\right).$ Take an element $\phi\in\left(\textup{Diff}\left(\CC,0\right)\right)^{r}$
and right compose $G^{i},\ i=1,2,$ by $\phi$ ; apply the process
of normalization following a sequence of $\Gamma-$reductions initially
fixed and truncate the final expressions. In the end, the coefficients
of the new normalized couple of generators 
\[
\left(\mathfrak{N}^{i}\left(G^{i}\circ\phi\right)\right)_{i=1,2}
\]
corresponds to some expressions $\phi\cdot A$ which depends on $A$
and $\phi.$
\begin{lem}
The application $\left(\phi,A\right)\to\phi\cdot A$ is an action.
More precisely, for any $\phi,\psi$ in $\left(\textup{Diff}\left(\CC,0\right)\right)^{r}$
\[
\phi\cdot\left(\psi\cdot A\right)=\left(\psi\circ\phi\right)\cdot A.
\]
\end{lem}
\begin{proof}
For $i=1,2,$ consider a normalized basis $\left(G^{1},G^{2}\right)$
and the two following normalizations 
\[
\left(\mathfrak{N}^{i}\left(G^{i}\circ\psi\circ\phi\right)\right)_{i=1,2}\textup{ and }\left(\mathfrak{N}^{i}\left(\mathfrak{N}^{i}\left(G^{i}\circ\psi\right)\circ\phi\right)\right)_{i=1,2}.
\]
Both are normalized bases of the ring 
\[
\left(\psi\circ\phi\right)^{\star}\mathcal{O}_{C}=\left\{ \gamma\circ\psi\circ\phi\left|\gamma\in\mathcal{O}_{C}\right.\right\} .
\]
The rings $\mathcal{O}_{C}$ and $\left(\psi\circ\phi\right)^{\star}\mathcal{O}_{C}$
share the same semiring of valuations $\Gamma.$ Thus, for $i=1,2$,
the valuation 
\[
\nu\left(\mathfrak{N}^{i}\left(G^{i}\circ\psi\circ\phi\right)-\mathfrak{N}^{i}\left(\mathfrak{N}^{i}\left(G^{i}\circ\psi\right)\circ\phi\right)\right)
\]
is an element of $\Gamma\cap\underset{l,i}{{\displaystyle \prod}}P_{l}^{i,\left(q_{i}\right)}.$
Since $\underline{P}^{i,\left(q_{i}\right)}$ is $\Gamma-$reduced,
this valuation is $\underline{\infty}$ and one has 
\[
\mathfrak{N}^{i}\left(G^{i}\circ\psi\circ\phi\right)=\mathfrak{N}^{i}\left(\mathfrak{N}^{i}\left(G^{i}\circ\psi\right)\circ\phi\right),
\]
which is the lemma.
\end{proof}
Let us denote by $\textup{Diff}^{c}\left(\CC,0\right)$ the quotient
of $\textup{Diff}\left(\CC,0\right)$ by the normal subgroup of elements
of the form
\[
t\to t+ut^{c}+\cdots.
\]
The truncation at $\sigma$ being part of the normalization process,
it follows that the previous action factorizes through 
\begin{equation}
\prod_{i=1}^{r}\textup{Diff}^{\sigma_{i}}\left(\CC,0\right).\label{group}
\end{equation}
Since the group (\ref{group}) is a connected solvable algebraic group,
Theorem \ref{theorem.structure.moduli} follows from Lemma \ref{LemmeIsoAnneau}
and the previous constructions.

\subsection{An example.}

Let $S$ be the curve $y\left(y^{2}-x^{3}\right)=0$. Figure (\ref{fig:Semiring-of-values})
shows the semiring $\Gamma_{S}.$ In this rather simple situation,
it can computed \emph{by hand. }In the general case, there exist algorithms
to compute the semiring of a curve with several components - see \cite{semigroup}.
\begin{figure}[H]
\begin{centering}
\includegraphics[scale=0.3]{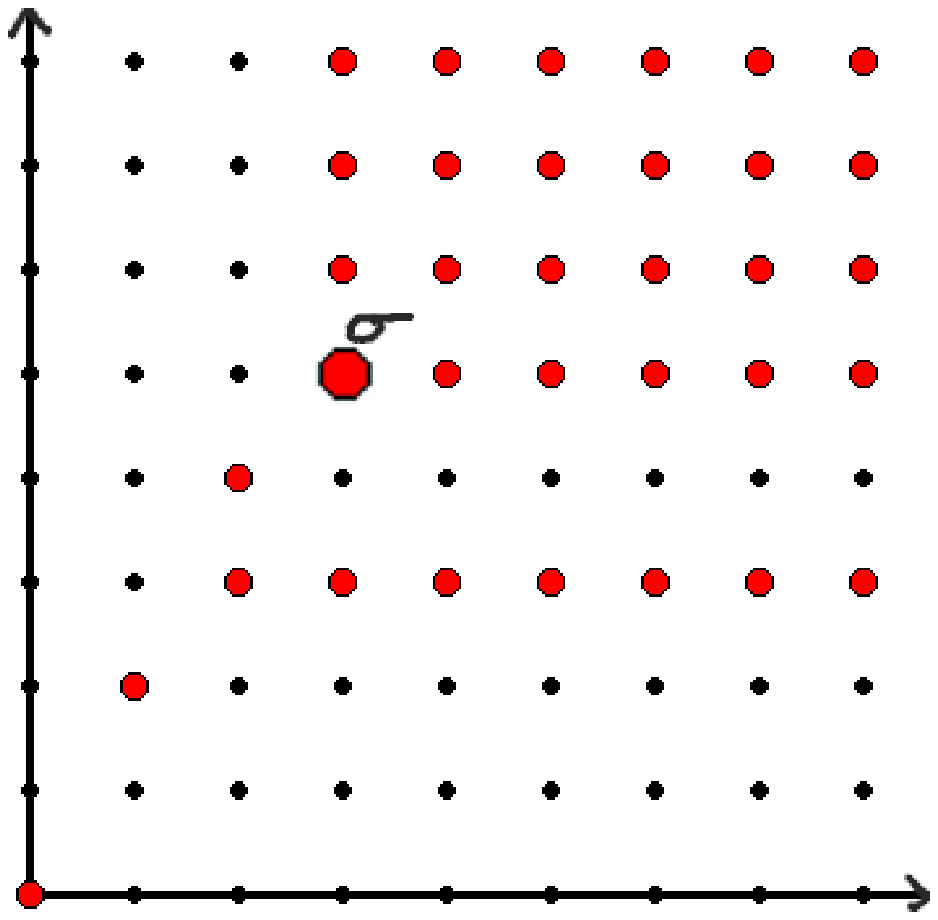}
\par\end{centering}
\caption{Se\label{fig:Semiring-of-values}miring of values the curve $\left\{ y\left(y^{2}-x^{3}\right)=0\right\} .$}
\end{figure}

Let $\left(C,\overline{h}\right)$ be in $\textup{Top}^{\bullet}\left(S\right)$.
The conductor $\sigma$ of $\Gamma$ is $\left(3,5\right)$. The set
of absolute points of $\Gamma$ is 
\[
\left\{ \left(1,2\right),\left(2,4\right),\left(3,\infty\right),\left(\infty,3\right)\right\} .
\]
Since $\left(2,4\right)=2\times\left(1,2\right)$, following \cite{Hernandes.semiring},
the set that minimaly generates $\Gamma$ as semiring is 
\[
g=\left\{ \left(1,2\right),\left(3,\infty\right),\left(\infty,3\right)\right\} .
\]

Since $\nu\left(x\right)=\left(1,2\right)$ and $\nu\left(y\right)=\left(\infty,3\right)$,
applying Lemma \ref{truncation} leads to a couple of generators that
are written 
\[
\left\{ G^{1}=\left(a_{11}^{1}t+a_{12}^{1}t^{2},a_{22}^{1}t^{2}+a_{23}^{1}t^{3}+a_{24}^{1}t^{4}\right),\ G^{2}=\left(0,a_{13}^{2}t^{3}+a_{14}^{2}t^{4}\right)\right\} 
\]
with $a_{11}^{1}\neq0,\ a_{22}^{1}\neq0,~a_{13}^{2}\neq0.$ Normalizing
some initial non vanishing coefficients provides the following couple
of generators
\[
\left\{ G^{1}=\left(t+a_{12}^{1}t^{2},a_{22}^{1}t^{2}+a_{23}^{1}t^{3}+a_{24}^{1}t^{4}\right),\ G^{2}=\left(0,t^{3}+a_{14}^{2}t^{4}\right)\right\} .
\]
To reduce $G^{1}$ we consider the following data
\[
P_{1}^{1,\left(0\right)}=\left\{ \infty,2\right\} \qquad P_{2}^{1,\left(0\right)}=\left\{ \infty,2,3,4\right\} ,
\]
and the two successive $\Gamma_{S}-$reductions defined as follows
\[
\begin{array}{ccl}
\underline{P}^{1,\left(0\right)} & \xrightarrow{\left(2,3\right),1} & \underline{P}^{1,\left(1\right)}=\left(\left\{ \infty\right\} ,\left\{ \infty,2,3,4\right\} \right).\\
\underline{P}^{1,\left(1\right)} & \xrightarrow{\left(\infty,3\right),2} & \underline{P}^{1,\left(2\right)}=\left(\left\{ \infty\right\} ,\left\{ \infty,2,4\right\} \right)
\end{array}
\]
Observe $\underline{P}^{1,\left(2\right)}$ is $\Gamma_{S}-$reduced.
Since $\nu\left(\left(G^{1}\right)^{2}+G^{2}\right)=\left(2,3\right)$,
the transformation associated to the first $\Gamma_{S}-$reduction
as in (\ref{eq:transformation}) is written
\[
G^{1}-a_{12}^{1}\left(\left(G^{1}\right)^{2}+G^{2}\right)=\left(t+t^{3}\left(\cdots\right),\left(\star\right)t^{2}+\left(\star\right)t^{3}+\left(\star\right)t^{4}+t^{5}\left(\cdots\right)\right),
\]
which leads to a new generator that we still denote by $G^{1}$. Noticing
that $\nu\left(G^{2}\right)=\left(\infty,3\right)$ yields the transformation
\[
G^{1}-\left(\star\right)G^{2}.
\]
The truncation at $\sigma=\left(3,5\right)$ finishes the normalization
of $G^{1}.$ The generator $G^{2}$ is already normalized since $\left(\infty,4\right)\notin\Gamma$.

Therefore, the normalized family $G$ has for final form 
\begin{equation}
\left\{ G^{1}=\left(t,at^{2}+bt^{4}\right),G^{2}=\left(0,t^{3}+ct^{4}\right)\right\} \label{eq:generators}
\end{equation}
and its elements depends only on $\mathfrak{E}_{\left(C,\overline{h}\right)}$.
The map $\mathbb{M}_{S}$ is defined by 
\[
\mathbb{M}_{S}:\left\{ \begin{array}{ccc}
\textup{Top}^{\bullet}\left(S\right) & \to & \mathbb{C}^{3}\\
\left(C,\overline{h}\right) & \mapsto & \left(a,b,c\right)
\end{array}\right.
\]
following (\ref{eq:generators}). By construction if $a\neq0,$ any
curve $C$ associated to a ring generated by such a family admits
a semiring of values $\Gamma_{C}$ that contains $\left(1,2\right)$
and $\left(\infty,3\right).$ It can be checked that $a\neq0$ is
the sole condition to ensure that actually, $\Gamma_{C}=\Gamma_{S}.$
Thus the image of $\mathbb{M}_{S}$ is the constructible set $\mathbb{C}\setminus\left\{ 0\right\} \times\CC^{2}.$

Let us compute the action of $\phi\in\textup{Diff}^{3}\left(\CC,0\right)\times\textup{Diff}^{5}\left(\CC,0\right)$
on $A\in\mathbb{C}\setminus\left\{ 0\right\} \times\mathbb{C}^{2}$
induced by the present construction where 
\[
\phi=\left(ut+vt^{2},\alpha t+\beta t^{2}+\gamma t^{3}+\Delta t^{4}\right)\textup{ and }A=\left(a,b,c\right).
\]
The action of $\phi$ on $A$ is written 
\begin{align*}
\phi\cdot A & =\left(\frac{a\alpha^{2}}{u},\frac{-a^{2}\alpha^{4}v-2a\alpha^{2}\beta cu^{2}+\alpha^{4}bu^{2}+2a\alpha\gamma u^{2}-5a\beta^{2}u^{2}}{u^{3}},\alpha c+3\frac{\beta}{\alpha}\right)
\end{align*}
and the quotient reduces to the class of the point $\left(1,0,0\right)$.
As a matter of fact, the curve $S$ has no moduli \cite{PaulGen}.

\section{Optimal vector field for a germ of curve $S$.}

The space $\moduli$ is now endowed with a complex structure. The
remainder of the article is interested in the generic dimension of
$\moduli$. In order to reach this purpose, subsequently, we proceed
to the study of the module of vector fields tangent to $S.$

Let $S$ be a germ of curve in $\left(\mathbb{C}^{2},0\right)$ and
$f$ a reduced equation of $S.$ Throughout this article, $\Der\left(\log S\right)$
will stand for be the $\mathcal{O}_{\left(\mathbb{C}^{2},0\right)}-$
module of vector fields tangent to $S,$that is such that the set
of vector fields $X$ such that 
\[
X\cdot f\in\left(f\right).
\]
It will be called \emph{the Saito module of $S$ }in reference to
\cite{MR586450}. Associated to the latter, we consider the following
analytical invariant
\begin{defn}
The \emph{Saito number} of $S$ is the integer 
\[
\mathfrak{s}\left(S\right)=\min_{X\in\Der\left(\log S\right)}\nu\left(X\right),
\]
where $\nu$ is the valuation defined by
\[
\nu\left(a\partial_{x}+b\partial_{y}\right)=\min\left(\nu\left(a\right),\nu\left(b\right)\right).
\]
\end{defn}
According to \cite{MR586450}, the Saito module of $S$ is a free
$\mathcal{O}_{\left(\mathbb{C}^{2},0\right)}-$ module of rank $2.$
If $\left\{ X_{1},X_{2}\right\} $ is one of its basis, said to be
\emph{a Saito basis for $S$}, it is easily seen that the number of
Saito of $S$ satisfies 
\[
\mathfrak{s}\left(S\right)=\min\left(\nu\left(X_{1}\right),\nu\left(X_{2}\right)\right).
\]
Following again \cite{MR586450}, $\left\{ X_{1},X_{2}\right\} $
is a Saito basis for $S$ if and only if the following property holds.
\begin{criterion*}[Criterion of Saito]
$\left\{ X_{1},X_{2}\right\} $ is a Saito basis for $S$ if and
only if there exists a germ of unit $u$ such that 
\begin{equation}
X_{1}\wedge X_{2}=uf,\label{Saito.fondamental.wedge}
\end{equation}
where $\cdot\wedge\cdot$ stands for determinant of the vector fields
in any coordinates.
\end{criterion*}
The property (\ref{Saito.fondamental.wedge}) will be referred to
as \emph{the criterion of Saito.} Evaluating the valuation of (\ref{Saito.fondamental.wedge})
gives the inequality
\begin{equation}
\nu\left(X_{1}\right)+\nu\left(X_{2}\right)\leq\nu\left(X_{1}\wedge X_{2}\right)=\nu\left(f\right)=\nu\left(S\right).\label{inegalite.fondamentale.saito}
\end{equation}
In particular, one has 
\begin{equation}
\mathfrak{s}\left(S\right)\leq\frac{\nu\left(S\right)}{2}.\label{saito.moitie.de.la.courbe}
\end{equation}

\begin{defn}
A vector field $X\in\Der\left(\log S\right)$ is said to be \emph{optimal}
for $S$ if $\nu\left(X\right)=\mathfrak{s}\left(S\right).$
\end{defn}
\begin{example}
Let $S$ be the double cusp given by 
\[
S=\left\{ \left(x^{2}-y^{3}\right)\left(y^{2}-x^{3}\right)=0\right\} .
\]
Then an optimal vector field can be given by 
\[
X=\left(2x^{2}+\frac{5}{2}y^{3}-\frac{9}{2}x^{3}y\right)\partial_{x}+\left(3xy-3x^{2}y^{2}\right)\partial_{y}.
\]
In particular 
\[
\mathfrak{s}\left(S\right)=2.
\]
\end{example}
\begin{prop}
If $X$ is optimal for $S$, then there exists a vector field $Y$
such that $\left\{ X,Y\right\} $ is a Saito basis for $S.$
\end{prop}
\begin{proof}
Let $\left\{ X_{1},X_{2}\right\} $ be any Saito basis for $S.$ According
to the criterion of Saito, there exists a unit $u$ such that 
\begin{equation}
X_{1}\wedge X_{2}=uf.\label{eq:saito.critere}
\end{equation}
Since $\left\{ X_{1},X_{2}\right\} $ is a basis, there exist functions
$u_{i},~i=1,2$ such that 
\[
X=u_{1}X_{1}+u_{2}X_{2}.
\]
Since $\nu\left(X\right)=\mathfrak{s}\left(S\right)=\min\left(\nu\left(X_{1}\right),\nu\left(X_{2}\right)\right)$,
for some $i$, say $i=1$, $u_{i}$ is a unit. Then, using (\ref{eq:saito.critere})
yields 
\[
X\wedge X_{2}=u_{1}uf.
\]
and thus, $\left\{ X,X_{2}\right\} $ is a Saito basis for $S.$
\end{proof}

\subsection{Curve of radial type.}

Let $E$ be the single blowing-up at $0$. The total space of the
blowing-up will be denoted by $\mathcal{M}$, 
\[
E:\left(\mathcal{M},D\right)\to\left(\mathbb{C}^{2},0\right).
\]
For any curve $S,$ $S^{E}$ will stand for the strict transform of
$S$ by $E$, that is the closure in $\mathcal{M}$ of $E^{-1}\left(S\setminus\left\{ 0\right\} \right).$
Moreover, for any vector field $Y,$ $Y^{E}$ will be the blown-up
vector field $E^{\star}Y$ divided by the maximal power of a local
equation of $D.$
\begin{defn}
Let $Y$ be a germ of vector field in $\left(\mathbb{C}^{2},0\right)$.
It is is said \emph{dicritical }if $Y^{E}$ is generically transverse
to the exceptional divisor $D$.
\end{defn}
A vector field $X$ being dicritical, it can be written in some coordinates
$\left(x,y\right)$
\[
X=R\left(x,y\right)\left(x\partial_{x}+y\partial_{y}\right)+\left(\cdots\right)
\]
where $R$ is an homogeneous polynomial function and $\left(\cdots\right)$
stands for higher order terms. Suppose that $X$ is optimal for $S$
and let $Y$ be such that $\left\{ X,Y\right\} $ is a basis of $\Der\left(\log S\right)$.
For any couple of non-vanishing functions $\left(a,b\right)$, the
initial part of $aX+bY$, that is, its homogeneous part of smallest
degree, is written
\[
a\left(0\right)R\left(x,y\right)\left(x\partial_{x}+y\partial_{y}\right)+b\left(0\right)Y^{\left(\mathfrak{s}\left(S\right)\right)}.
\]
where $Y^{\left(\star\right)}$ stands for the homogeneous part of
degree $\star$ of $Y.$ If $Y$ is optimal and not dicritical then
for $a$ and $b$ generic, $aX+bY$ is not dicritical. Which is why,
we consider the following definition
\begin{defn}
$S$ is said to be \emph{of radial type} if all optimal vector fields
for $S$ are dicritical.
\end{defn}

\subsection{Flat Saito basis.}

In this section, we are going to identify an open dense set $\text{\ensuremath{U}}\subset\moduli$
for which, the Saito basis of $C\in U$, can be extended locally around
$C$ in $\moduli$ into a family of Saito bases. Further on, an example
will illustrate that this property holds only generically.
\begin{thm}
\label{Flat.saito.basis}There exist an open dense set $U\subset\moduli$
on which the Saito number is constant. More precisely, for any $\left(C,\overline{h}\right)\in U$,
there exists a germ of analytical family of vector fields 
\[
c\in\left(\moduli,\left(C,\overline{h}\right)\right)\mapsto X_{i}\left(c\right),\ i=1,2
\]
 such that for any $c$, the family $\left\{ X_{1}\left(c\right),X_{2}\left(c\right)\right\} $
is a Saito basis for $c$ with 
\[
\forall c\in\left(\moduli,\left(C,\overline{h}\right)\right),\ \nu\left(X_{1}\left(c\right)\right)=\mathfrak{s}\left(c\right).
\]
\end{thm}
\begin{proof}
Let $\left(C,\overline{h}\right)\in\moduli$ be a regular point for
the complex structure of $\moduli$. Consider a miniversal deformation
of $C$ 
\begin{equation}
\left(\Sigma,C\right)\subset\left(\mathbb{C}^{2+N},0\right)\xrightarrow{\pi}\left(\mathbb{C}^{N},0\right),\qquad\pi\left(x,t\right)=t\in\mathbb{C}^{N}\label{Deformation}
\end{equation}
versal for topologically trivial deformations of $C$ and for which
the singular locus of $\Sigma$ is $\left\{ 0\right\} \times\mathbb{C}^{N}:$
it is enough to consider the miniversal deformation of any reduced
equation of $C$ and to restrict it to the associated smooth $\mu-$constant
stratum. We fix an open neighborhood $\mathbb{C}^{2+N}\supset\mathcal{U}\ni0$
on which $\Sigma$ and $C$ are well defined. By shrinking $\mathcal{U}$
if necessary, we can also suppose that, out of its singular locus,
$\Sigma$ is transverse to the fiber of $\pi,$ that is for any $p\in\mathcal{U}\setminus\left\{ 0\right\} \times\mathbb{C}^{N},$
\begin{equation}
\pi^{-1}\left(\pi\left(p\right)\right)\not\subset T_{p}\Sigma\label{transversality}
\end{equation}

The deformation (\ref{Deformation}) is topologically trivial : more
precisely, there exists an homeomorphism $\mathcal{H}$:$\left(\mathbb{C}^{2+N},0\right)\to\left(\mathbb{C}^{2+N},0\right)$
such that
\begin{enumerate}
\item $\pi\mathcal{H=\pi}$
\item $\left.\mathcal{H}\right|_{\pi^{-1}\left(0\right)}=h$
\item The following diagram commutes\[\xymatrix{ \left(S\times \left(\mathbb{C}^{N},0\right),S\right)  \ar[rd]^{\pi} \ar[r]^-{\mathcal{H}}    &   \left(\Sigma,C\right)  \ar[d]^\pi \\    & \left(\mathbb{C}^{N},0\right) } \]
\end{enumerate}
By construction, the map $\mathfrak{C}$ defined by 
\begin{equation}
t\in\left(\mathbb{C}^{N},0\right)\xmapsto{\mathfrak{C}}\left(\left.\Sigma\right|_{\pi^{-1}\left(t\right)},\left.\overline{\mathcal{H}}\right|_{\pi^{-1}\left(t\right)}\right)\in\moduli\label{local.chart}
\end{equation}
is a local diffeomorphism.

For technical reason, we add to $\Sigma$ an hyperplane $H$ not contained
in $\Sigma$ and transverse to $\pi.$ Consider $\Sigma^{\circ}=\Sigma\cup H$$.$
In what follows, $f_{\Sigma^{\circ}}$ stands for a reduced equation
of $\Sigma$. The kernel of the evaluation map 
\[
\text{Der}\left(\log\Sigma^{\circ}\right)\xrightarrow{d\pi\left(\cdot\right)}\left(\mathcal{O}_{N+2}\right)^{N}
\]
is the sheaf $\text{Der}^{\uparrow}\left(\log\Sigma^{\circ}\right)$
of \emph{vertical} vector fields tangent to $\Sigma^{\circ}$. In
the initial coordinates $\left(x,y,t\right)$ a section of $\text{Der}^{\uparrow}\left(\log\Sigma^{\circ}\right)$
is written 
\[
a\left(x,y,t\right)\partial_{x}+b\left(x,y,t\right)\partial_{y}
\]
where $a$ and $b$ are analytic functions. The sheaf $\text{Der}\left(\log\Sigma^{\circ}\right)$
is coherent, so is $\text{Der}^{\uparrow}\left(\log\Sigma^{\circ}\right).$
Note that if $X$ is a section of $\text{Der}^{\uparrow}\left(\log\Sigma^{\circ}\right)$,
then for any $t\in\pi\left(\mathcal{U}\right)$, $\left.X\right|_{\pi^{-1}\left(t\right)}$
is tangent to $\left.\Sigma^{\circ}\right|_{\pi^{-1}\left(t\right)}.$
Fix a system of generators 
\begin{equation}
\left\{ X_{1},\cdots,X_{n}\right\} \label{minimal.system.generateur}
\end{equation}
of $\text{Der}^{\uparrow}\left(\log\Sigma^{\circ}\right)\left(\mathcal{U}\right).$
We are going to use the following remarks which are consequences of
the coherence property : for any open set $\mathcal{V}\subset\mathcal{U}$,
the vector fields $\left.X_{1}\right|_{\mathcal{V}},\cdots,\left.X_{n}\right|_{\mathcal{V}}$
generate $\text{Der}^{\uparrow}\left(\log\Sigma^{\circ}\right)\left(\mathcal{V}\right).$
Moreover,
\begin{enumerate}
\item if $\mathcal{V}$ does not meet $\Sigma^{\circ}$ then $\text{Der}^{\uparrow}\left(\log\Sigma^{\circ}\right)\left(\mathcal{V}\right)$
is the set of all holomorphic vertical vector fields on $\mathcal{V}.$
\item if $\mathcal{V}$ meets the smooth part of $\Sigma^{\circ}$, then
$\text{Der}^{\uparrow}\left(\log\Sigma^{\circ}\right)\left(\mathcal{V}\right)$
is freely generated by the vertical vector fields $x\partial_{x}$
and $\partial_{y}$ where $\left(x,y,t\right)$ is a local system
of coordinates preserving the fibration $\pi$ for which $x=0$ is
an equation of the trace of $\Sigma^{\circ}$ on $\mathcal{V}$: such
a local system of coordinates exists under the transversality property
(\ref{transversality}). In particular, the product 
\[
x\partial_{x}\wedge\partial_{y}
\]
vanishes at order $1$ along $\Sigma^{\circ}.$
\end{enumerate}
All the $X_{i}'s$ cannot vanish identically on a given component
of $\Sigma^{\circ}$ because for instance the section of $\text{Der}^{\uparrow}\left(\log\Sigma^{\circ}\right)$
defined by
\[
\partial_{x}\left(f_{\Sigma^{\circ}}\right)\partial_{y}-\partial_{y}\left(f_{\Sigma^{\circ}}\right)\partial_{x}
\]
does not vanish on any component of $\Sigma^{\circ}.$ Considering
if necessary a combination of the $X_{i}'s$, we can suppose that
$X_{1}$ does not vanish identically on any component of $\Sigma^{\circ}.$
We can also suppose that $X_{1}$ is singular in codimension $2$
: indeed, if not, there exists $\tilde{X}_{1}$ such that $X_{1}=h\tilde{X}_{1}$
where $h$ is an holomorphic map with $h\left(0\right)=0.$ Since
$h$ cannot vanish identically on any component of $\Sigma^{\circ}$,
$\tilde{X}_{1}$ is tangent to $\Sigma^{\circ}$ and the family 
\[
\left\{ \tilde{X}_{1},\cdots,X_{n}\right\} 
\]
still generates the sheaf $\text{Der}^{\uparrow}\left(\log\Sigma^{\circ}\right)$.
Now, if there exists $j\neq1$ such that 
\[
X_{1}\wedge X_{j}\equiv0
\]
then, by division, there exists $\phi$ such that $X_{j}=\phi X_{1}$,
which contradicts the minimality of the system of generators (\ref{minimal.system.generateur}).
Thus, for any $j\neq1$, there exists a function $g_{j}\not\equiv0$
such that
\begin{equation}
X_{1}\wedge X_{j}=f_{\Sigma^{\circ}}g_{j}=xf_{\Sigma}g_{j}\label{wedge.generateur}
\end{equation}
where the system of coordinates $\left(x,y,t\right)$ is chosen so
that $x$ is an equation of the added hyperplan $H.$ Consider a point
$p$ in the zero set $Z\left(g_{2},\ldots,g_{n}\right)$ of the ideal
$\left(g_{2},\ldots,g_{n}\right)$. If $p$ is not in $\Sigma^{\circ}$
then all the generators of $\text{Der}^{\uparrow}\left(\log\Sigma^{\circ}\right)$
are tangent two by two at $p$, which is impossible in view of the
above remark $\left(1\right)$. Therefore, one has 
\[
Z\left(g_{2},\ldots,g_{n}\right)\subset\Sigma^{\circ}.
\]
We are going to improve the above inclusion, showing that one can
suppose that
\[
Z\left(g_{2},\ldots,g_{n}\right)\subset\left\{ 0\right\} \times\mathbb{C}^{N}.
\]
Consider the following set 
\[
\Delta=\left\{ \left.t\in\left(\mathbb{C}^{N},0\right)\right|\pi^{-1}\left(t\right)\subset Z\left(g_{2},\ldots,g_{n}\right)\right\} .
\]
It is a closed analytic subset of $\left(\mathbb{C}^{N},0\right)$
and we remove $\pi^{-1}\left(\Delta\right)$ of $\mathcal{U}.$ Now,
fixed some $t$ and denote by denote by $I$ the canonical injection
$I:\left(\mathbb{C}^{2},0\right)\to\pi^{-1}\left(t\right),$ $I\left(x\right)=\left(x,t\right)$.
If the intersection
\[
\Delta_{t}=\left.\Sigma^{\circ}\right|_{\pi^{-1}\left(t\right)}\cap Z\left(g_{2},\cdots,g_{n}\right)
\]
contains $0\times\left\{ t\right\} $ has a non isolated point of
$\Delta_{t}$, it contains also an analytic curve which is a component
of $I^{\star}f_{\Sigma^{\circ}}=0$. Therefore there is a factor $h$
of $I^{\star}f_{\Sigma^{\circ}}$ that divides $I^{\star}g_{i}$ for
any $i.$ Thus, for $i\geq2$, one has
\[
\left.X_{1}\right|_{\pi^{-1}\left(t\right)}\wedge\left.X_{i}\right|_{\pi^{-1}\left(t\right)}=h^{2}\left(\cdots\right).
\]
Since the vector fields $\left.X_{i}\right|_{\pi^{-1}\left(t\right)}$
are tangent to $h=0$, any couple of element in $\left.\text{Der}^{\uparrow}\left(\log\Sigma^{\circ}\right)\right|_{\pi^{-1}\left(t\right)}$
has a contact of order $2$ locally around the zero locus of $h$,
which is impossible according to the above remark $\left(2\right)$.
As a consequence, for any $t\in\text{\ensuremath{\left(\mathbb{C}^{N},0\right)}},$
if $\Delta_{t}$ contains $0\times\left\{ t\right\} $, it is as an
isolated point in $\Delta_{t}$. So, $\Delta_{t}$ is a finite set.
\begin{lem}
Let $W\subset\left(\mathbb{C}^{2+N},0\right)$ be an analytic set
such that for any $t$, $W\cap\pi^{-1}\left(t\right)$ is finite.
Then 
\[
\left\{ 0\right\} \times\mathbb{C}^{N}\not\subset\overline{W\setminus\left\{ 0\right\} \times\mathbb{C}^{N}}.
\]
\end{lem}
\begin{proof}
The hypothesis ensures that $\textup{codim}W\geq2.$ If $\textup{codim}W\geq3,$
the lemma is clear since $\textup{codim}\left(\underbrace{\left\{ 0\right\} \times\mathbb{C}^{N}}_{V}\right)=2$.
Suppose $\textup{codim}W=2$. Let us write 
\[
V=\left(W\cap V\right)\cup\left(\overline{V\setminus W}\right).
\]
Since $V$ is irreducible, either $W\cap V=V$, and $V$ is an irreducible
component of $W$, or $\overline{V\setminus W}=V,$ and $V\cap W$
is an analytic subset of $V$ of codimension at least $1$ in $V$,
and thus of codimension at least $3$ in $\mathbb{C}^{2+N}.$ In any
case, the lemma is proved.
\end{proof}
Following the lemma, the analytic set $K$ defined by 
\[
K=\left\{ 0\right\} \times\mathbb{C}^{N}\cap\overline{\Sigma^{\circ}\cap Z\left(g_{2},\cdots,g_{n}\right)\setminus\left\{ 0\right\} \times\mathbb{C}^{N}}
\]
is a strict analytic subset of $\left\{ 0\right\} \times\mathbb{C}^{N}$
such that $\left\{ 0\right\} \times\mathbb{C}^{N}\setminus K$ admits
a neighborhood $\mathcal{U}$ in $\left(\mathbb{C}^{2+N},0\right)$
satisfying 
\[
Z\left(g_{2},\cdots,g_{n}\right)\cap\mathcal{U}\subset\left\{ 0\right\} \times\mathbb{C}^{N}
\]
At the level of the ideals, the inclusion above ensures that there
exists $M\in\mathbb{N}$ such that 
\[
\left(x,y\right)^{M}\subset\left(g_{2},\ldots,g_{n}\right)
\]
As a consequence, there exists a relation of the following form 
\[
x^{M}=\sum_{i=2}^{n}h_{i}g_{i}
\]
and considering $Y=\sum_{i=2}^{n}h_{i}X_{i}$ and the relation (\ref{wedge.generateur})
yields a vector field $Y$ in $\text{Der}^{\uparrow}\left(\log\Sigma^{\circ}\right)\left(\mathcal{U}\right)$
such that 
\[
X_{1}\wedge Y=f_{\Sigma}x^{M+1}.
\]
Notice that $X_{1}$ and $Y$ are both tangent to $x=0.$ Let us write
in coordinates 
\begin{align*}
X_{1} & =xa^{1}\left(x,y,t\right)\frac{\partial}{\partial x}+\left(b_{0}^{1}\left(y,t\right)+xb_{1}^{1}\left(x,y,t\right)\right)\frac{\partial}{\partial y}\\
Y & =xa^{2}\left(x,y,t\right)\frac{\partial}{\partial x}+\left(b_{0}^{2}\left(y,t\right)+xb_{1}^{2}\left(x,y,t\right)\right)\frac{\partial}{\partial y}.
\end{align*}
Replacing if necessary $X_{1}$ by $X_{1}+Y$, we can suppose that
\[
\nu_{y}\left(b_{0}^{1}\right)\leq\nu_{y}\left(b_{0}^{2}\right)
\]
where $\nu_{y}$ is the valuation in the ring $\mathbb{C}\left\{ t\right\} \left\{ y\right\} $.
Consider the vertical vector field 
\[
\tilde{Y}=\frac{1}{x}\left(Y-\frac{b_{0}^{2}}{b_{0}^{1}}X_{1}\right).
\]
It is holomorphic removing if necessary, some fibers $\pi^{-1}\left(t\right)$
for $t$ in some closed analytic set of $\mathbb{C}^{N}$ related
to the zeros of $b_{0}^{1}\left(0,t\right)$. Moreover, one has 
\[
X_{1}\wedge\tilde{Y}=f_{\Sigma}x^{M}.
\]
Since $X_{1}$ is tangent to $x=0$ and since its singular locus has
codimension $2$, $\tilde{Y}$ is also tangent to $x=0$. The process
can be repeated and finally, one obtains two vertical vector fields
$X_{1}$ and $X_{2}$ tangent to $\Sigma$ such that 
\[
X_{1}\wedge X_{2}=f_{\Sigma}.
\]
The functions
\[
t\in\pi\left(\mathcal{U}\right)\mapsto\nu\left(\left.X_{i}\right|_{\pi^{-1}\left(t\right)}\right),\ i=1,2
\]
are lower semi-continuous. Replacing if necesssary $X_{1}$ by $X_{1}+X_{2},$
we consider an open set $\mathcal{U^{\prime}}\subset\left(\mathcal{\mathbb{C}}^{N},0\right)$,
whose closure is a neighborhood of $0$, on which 
\[
\forall t\in\mathcal{\mathcal{U^{\prime}}},\ \nu\left(\left.X_{1}\right|_{\pi^{-1}\left(t\right)}\right)\leq\nu\left(\left.X_{2}\right|_{\pi^{-1}\left(t\right)}\right).
\]
According to the criterion of Saito, for any $t$, 
\[
\left\{ \left.X_{1}\right|_{\pi^{-1}\left(t\right)},\left.X_{2}\right|_{\pi^{-1}\left(t\right)}\right\} 
\]
consists in basis of Saito for the curve $\left.\Sigma\right|_{\pi^{-1}\left(t\right)}.$
Therefore, for any $t\in\mathcal{U}^{\prime},$ one has 
\[
\nu\left(\left.X_{1}\right|_{\pi^{-1}\left(t\right)}\right)=\mathfrak{s}\left(\left.\Sigma\right|_{\pi^{-1}\left(t\right)}\right).
\]
From (\ref{local.chart}), one can consider the open set $\mathfrak{C}\left(\mathcal{U}^{\prime}\right)\subset\moduli$
and the union of such open sets while the above construction is done
in the neighborhood of any regular point $\left(C,\overline{h}\right)$
in $\moduli.$ By construction, the resulting open set has the desired
properties.
\end{proof}
From now on, a curve $\text{C}$ in $\moduli$ will be said \emph{generic}
if it belongs to the open set identified in the theorem above : in
that sense, for a generic curve $C$ in its moduli space, we will
be allowed to consider a analytical\emph{ }family of Saito bases following
any topologically trivial deformation of $C.$
\begin{example}
Consider the union of four regular transversal curves. Up to some
change of coordinates, it can be written 
\[
S=\left\{ xy\left(y+x\right)\left(y+t_{1}x\right)=0\right\} 
\]
where $t_{1}\in\moduli=\mathbb{C}\setminus\left\{ 0,1\right\} .$
It can be seen \cite{MR2808211} that it admits a miniversal deformation
for the topologically trivial deformations of the form
\[
\Sigma=\left\{ F\left(x,y,t\right)=xy\left(y+x\right)\left(y+tx\right)=0\right\} \in\left(\mathbb{C}^{2}\times\mathbb{C},\left(0,0,t_{1}\right)\right).
\]
The basis highlighted in Theorem \ref{Flat.saito.basis} can be explicited
in the above coordinates as 
\[
X_{1}=x\partial x+y\partial y,\qquad X_{2}=\partial_{x}F\partial_{y}-\partial_{y}F\partial_{x}.
\]
In this case, $X_{1}$ and $X_{2}$ is a Saito basis in a whole neighborhood
of $t_{1}\in\moduli.$ In general, the situtation is not so favourable.
\end{example}
\begin{example}
\label{Example.5.droites}Consider for instance the union of five
regular transversal curves, which is written 
\[
S=\left\{ xy\left(y+x\right)\left(y+\alpha x\right)\left(y+\beta x\right)=0\right\} .
\]
with $\alpha\neq0,1$ and $\beta\neq0,1,\alpha$. A miniversal deformation
of $S$ is written 
\[
\Sigma=\left\{ F=xy\left(y+x\right)\left(y+t_{1}x\right)\left(y+t_{2}x+t_{3}x^{2}\right)=0\right\} \in\left(\mathbb{C}^{2}\times\mathbb{C}^{3},\left(0,0,\alpha,\beta,0\right)\right).
\]
For the curve $S$, which corresponds to the parameter $\left(\alpha,\beta,0\right)$,
a basis of Saito is given by
\[
X_{1}=x\partial x+y\partial y,\qquad X_{2}=\partial_{x}F\partial_{y}-\partial_{y}F\partial_{x}.
\]
However, this basis cannot be \emph{extended,} in a whole neighborhood
of $\left(\alpha,\beta,0\right).$ Since $X_{1}$ has a valuation
equal to one, the Saito number of $S$ is equal to $1.$ It can be
seen that for any $t_{3}\neq0,$ the number of Saito of $\left.\Sigma\right|_{t=\left(\alpha,\beta,t_{3}\right)}$
is bigger than $2$. Indeed, consider a vector field $X$ tangent
to $\left.\Sigma\right|_{t=\left(\alpha,\beta,t_{3}\right)}.$ If
its valuation is smaller than $1$, then it is dicritical. Thus it
is written 
\[
X=k\left(x\partial_{x}+y\partial y\right)+\left(\cdots\right)
\]
where $k$ is a non vanishing constant. Following \cite{diff}, $X$
is linearizable and in some coordinates in which $X$ is linear, the
curve $\left.\Sigma\right|_{t=\left(\alpha,\beta,t_{3}\right)}$ becomes
exactly the union of five germs of straight lines, which is impossible
if $t_{3}\neq0.$ Finally, it can be seen that if $t_{3}\neq0$ then
\[
\mathfrak{s}\left(\left.\Sigma\right|_{t=\left(\alpha,\beta,t_{3}\right)}\right)=2
\]
and an optimal vector field for $\left.\Sigma\right|_{t=\left(\alpha,\beta,v\right)}$
is written
\[
X=\left(x+\epsilon y\right)\left(x\partial_{x}+y\partial y\right)+\left(\cdots\right)
\]
where $\epsilon\neq0,1.$
\end{example}

\subsection{Saito basis for $S$ and $S\cup l$}

The process described below allows us to obtain a Saito basis for
$S$ from a Saito basis for $S\cup l$ where the curve $l$ is a regular
curve. This trick has been already introduced in the proof of Theorem
\ref{Flat.saito.basis}. Throughout this article, it will be often
a key argument.

Let $S$ be a germ of curve and $l$ be a germ of smooth curve that
is not a component of $S$. Let $\left\{ X_{1},X_{2}\right\} $ be
a Saito basis for $S\cup l$. The Saito criterion is written 
\begin{equation}
X_{1}\wedge X_{2}=ufL\label{eq:critertion}
\end{equation}
where $u$ is a unity, $f$ a reduced equation of $S$ and $L$ a
reduced equation of $l.$ Let us consider a local system of coordinates
$\left(x,y\right)$ in which $L=x.$ Then, for $i=1,2,$ the vector
fields $X_{i}$ can be written 
\[
X_{i}=xa_{i}\partial_{x}+\left(b_{i}^{0}+xb_{i}^{1}\right)\partial_{y},\quad a_{i},\ b_{i}^{1}\in\mathbb{C}\left\{ x,y\right\} ,\ b_{i}^{0}\in\mathbb{C}\left\{ y\right\} .
\]
Considering if necessary a generic change of basis 
\[
\left\{ \alpha X_{1}+\beta X_{2},uX_{1}+vX_{2}\right\} 
\]
where $\left|\begin{array}{cc}
\alpha & \beta\\
u & v
\end{array}\right|\neq0$, one can suppose that 
\[
\nu\left(X_{i}\right)=\mathfrak{s}\left(S\cup l\right)\textup{ and }\nu_{y}\left(b_{1}^{0}\left(y\right)\right)=\nu_{y}\left(b_{2}^{0}\left(y\right)\right)
\]
where $\nu_{y}$ is the valuation in the ring $\mathbb{C}\left\{ y\right\} $.
In particular, the quotient $\frac{b_{1}^{0}}{b_{2}^{0}}$ extends
holomorphically at $\left(x,y\right)=\left(0,0\right)$ as a unit.
The relation \ref{eq:critertion} leads to 
\begin{equation}
\underbrace{\frac{\left(X_{1}-\frac{b_{1}^{0}}{b_{2}^{0}}X_{2}\right)}{x}}_{X_{1}^{\prime}}\wedge X_{2}=uf,\label{eq:criterion2}
\end{equation}
where $X_{1}^{\prime}$ extends holomorphically at $\left(0,0\right).$
Since $L=0$ is not a component of $S$, the vector field $X_{1}^{\prime}$
leaves invariant $S.$ The Saito criterion ensures that $\left\{ X_{1}^{\prime},X_{2}\right\} $
is a Saito basis for $S$.

Now, it is clear that
\[
\nu\left(X_{1}^{\prime}\right)\geq\nu\left(X_{1}\right)-1=\mathfrak{s}\left(S\cup l\right)-1.
\]
Since, $\nu\left(X_{2}\right)=\mathfrak{s}\left(S\cup l\right)$,
one has 
\[
\mathfrak{s}\left(S\right)=\mathfrak{s}\left(S\cup l\right)-1\textup{ or }\mathfrak{s}\left(S\cup l\right).
\]

Assume moreover, that $S$ is not of radial type but $S\cup l$ is.
By definition, $X_{1}$ and $X_{2}$ are dicritical. Thus, the homogeneous
part of degree $\mathfrak{s}\left(S\cup l\right)$ of $X_{i}$ is
written 
\[
X_{i}^{\left(\mathfrak{s}\left(S\cup l\right)\right)}=R_{i}\left(x\partial_{x}+y\partial_{y}\right)
\]

Therefore the homogeneous part of degree $\mathfrak{s}\left(S\cup l\right)-1$
of $X_{1}^{\prime}$ is
\begin{equation}
\frac{1}{x}\left(R_{1}-\frac{b_{1}^{0}}{b_{2}^{0}}\left(0\right)R_{2}\right)\left(x\partial_{x}+y\partial_{y}\right).\label{eq:radialtype}
\end{equation}
If the above expression does not identically vanish, then $X_{1}^{\prime}$
would be dicritical. Since $X_{2}$ is dicritical too, $S$ would
be of radial type, which is impossible. Thus, the expression \ref{eq:radialtype}
vanishes and $\nu\left(X_{1}^{\prime}\right)\geq\mathfrak{s}\left(S\cup l\right).$
Since $\nu\left(X_{2}\right)=\mathfrak{s}\left(S\cup l\right)$ one
has finally
\[
\mathfrak{s}\left(S\right)=\mathfrak{s}\left(S\cup l\right).
\]

Gathering the remarks above, we obtain the
\begin{prop}
\label{division.saito.process}Let $l$ be a germ of smooth curve
that is not a component of $S.$ Then
\begin{enumerate}
\item In any case, $\mathfrak{s}\left(S\right)=\mathfrak{s}\left(S\cup l\right)-1\textup{ or }\mathfrak{s}\left(S\cup l\right)$.
\item If $S$ is not of radial type but $S\cup l$ is then 
\[
\mathfrak{s}\left(S\right)=\mathfrak{s}\left(S\cup l\right).
\]
\end{enumerate}
\end{prop}
The process described above can be reversed. Consider a Saito basis
$\left\{ X_{1},X_{2}\right\} $ for $S.$ Changing of basis if necessary,
one can consider that 
\[
\nu\left(X_{1}\right)=\nu\left(X_{2}\right).
\]
Let $l$ be a \emph{generic} smooth curve and $L$ a reduced equation
of $l$. Fix some coordinates $\left(x,y\right)$ in which $l$ has
a parametrization of the form 
\[
\gamma\left(t\right)=\left(t,\epsilon\left(t\right)\right),\ t\in\left(\mathbb{C},0\right).
\]
The product 
\[
X_{1}\left(\gamma\right)\wedge\gamma^{\prime}\in\mathbb{C}\left\{ t\right\} 
\]
has a valuation in $\mathbb{C}\left\{ t\right\} $ equal to $\nu\left(X_{1}\right)$
or $\nu\left(X_{1}\right)+1$ depending on whether or not $X_{1}$
is dicritical. Therefore, the quotient 
\[
\frac{X_{1}\left(\gamma\right)\wedge\gamma^{\prime}}{X_{2}\left(\gamma\right)\wedge\gamma^{\prime}}
\]
extends holomorphically at $t=0$ as a unit $\phi\left(t\right).$
By construction, the vector field 
\[
X_{1}-\phi\left(x\right)X_{2}
\]
is tangent to the curve $l$. Finally, in the coordinates $\left(x,y\right)$,
according to the criterion of Saito, the family 
\begin{equation}
\left\{ X_{1}-\phi\left(x\right)X_{2},LX_{2}\right\} \label{ajouter.une.courbe}
\end{equation}
is a Saito basis for $S\cup l$.

\section{Generic element in $\textup{Der}\left(\log S\right)$ and adapted
Saito Bases.}

\global\long\def\dd{\textup{d}}%

In (\ref{saito.moitie.de.la.courbe}), we remark that 
\[
\mathfrak{s}\left(S\right)\leq\frac{\nu\left(S\right)}{2}.
\]
In this section, we will prove that for a curve $S$ generic in its
moduli space the latter inequality is essentially reached, as it will
be stated in Theorem \ref{thm:LowerGenericBound}.

\subsection{Generic value of $\mathfrak{s}\left(S\right)$\label{subsection.Generic.value}}

Let $S$ be a curve generic in its topological class and $\left\{ X_{1},X_{2}\right\} $
be a Saito basis for $S$.

The exceptional divisor $D$ of the blowing-up $E:\left(\mathcal{M},D\right)\to\left(\mathbb{C}^{2},0\right)$
can be covered by two open sets $U_{1}$ and $U_{2}$ and two charts
$\left(x_{1},y_{1}\right)$ and $\left(x_{2},y_{2}\right)$ defined
respectively in some neighborhoods of $U_{1}$ and $U_{2}$ such that
\[
E\left(x_{1},y_{1}\right)=\left(x_{1},y_{1}x_{1}\right)\qquad\textup{ and }\qquad E\left(x_{2},y_{2}\right)=\left(x_{2}y_{2},y_{2}\right).
\]
Let $\Theta_{S}$ be the sheaf on $\mathcal{M}$ of vector fields
tangent to $E^{\left(-1\right)}\left(S\right)=S^{E}\cup D$. Let $\omega$
be a $1-$form with an isolated singularity tangent to the vector
field $X_{1}$ : if $X_{1}$ is written 
\[
X_{1}=a\partial_{x}+b\partial_{y},
\]
one can choose 
\[
\omega=a\dd y-b\dd x.
\]
Let us consider the global $1-$form on $\mathcal{M}$ defined by
the pull-back 
\[
\Omega=E^{\star}\omega.
\]
We denote by $\mathfrak{B}$ the \emph{basic operator} : this is a
morphism of sheaves 
\[
\mathfrak{B}:\Theta_{S}\rightarrow\Omega^{2}\left(\mathcal{M}\right)
\]
that is written 
\[
\mathfrak{B}\left(T\right)=L_{T}\Omega\wedge\Omega=\Omega\left(T\right)d\Omega-d\left(\Omega\left(T\right)\right)\wedge\Omega.
\]
Here, $\Omega^{2}\left(\mathcal{M}\right)$ is the sheaf on $\mathcal{M}$
of holomorphic $2$-forms and $L_{T}$ is the Lie deriviative with
respect to the vector field $T.$ Following \cite{MR704017}, the
kernel of $\mathfrak{B}$ consists in the infinitesimal generators
of the sheaf of automorphisms of the foliation induced by $X_{1}^{E},$
that is, 
\[
L_{T}\Omega\wedge\Omega\equiv0\Longrightarrow\forall t\in\left(\mathbb{C},0\right),\ \left(\left(e^{tT}\right)^{\star}X_{1}^{E}\right)\wedge X_{1}^{E}\equiv0.
\]

The lemma below describes partially the image of $\mathfrak{B}.$
\begin{lem}
\label{lemma.inclusion.sheaves}$\mathfrak{B}\left(\Theta_{S}\right)\subset\Omega^{2}\left(-\overline{n}D-S^{E}\right)$
where
\begin{itemize}
\item $\overline{n}=2\nu\left(X_{1}\right)+\left\{ \begin{array}{cl}
2 & \textup{ if \ensuremath{X_{1}} is dicritical}\\
1 & \textup{ if not}
\end{array}\right.$
\item $\Omega^{2}\left(-\overline{n}D-S^{E}\right)$ is the sheaf of $2-$forms
that vanish along $D$ and $S^{E}$ with at least respective orders
$\overline{n}$ and $1.$
\end{itemize}
\end{lem}
\begin{proof}
It is a computation which can be performed in local coordinates. If
$X_{1}$ is dicritical, then out of $\textup{Sing}\left(X_{1}^{E}\right)$,
one can write 
\[
\Omega=ux_{1}^{\nu\left(X_{1}\right)+1}\dd y_{1},
\]
where $u$ is a local unit. A section $T$ of $\Theta_{S}$ is written
\[
T=\alpha x_{1}\partial_{x_{1}}+\beta\partial_{y_{1}},\quad\alpha,\beta\in\mathbb{C}\left\{ x_{1},y_{1}\right\} 
\]
Thus, applying the morphism $\mathfrak{B}$ yields 
\[
\mathfrak{B}\left(T\right)=\left(u^{2}x_{1}^{2\nu\left(X_{1}\right)+2}\partial_{x_{1}}\beta\right)\dd x_{1}\wedge\dd y_{1}.
\]
If $X_{1}$ is not dicritical, then out of the locus of tangency between
$X_{1}^{E}$ and $D$, one can write in some coordinate 
\[
\Omega=ux_{1}^{\nu\left(X_{1}\right)}\dd x_{1},
\]
and 
\[
\mathfrak{B}\left(T\right)=\left(u^{2}x_{1}^{2\nu\left(X_{1}\right)+1}\partial_{y_{1}}\alpha\right)\dd x_{1}\wedge\dd y_{1}.
\]
Finally, along a regular point of $S^{E},$ one can write 
\[
\Omega=u\dd y_{1},
\]
where $y_{1}=0$ is a local equation of $S^{E}.$ A local section
of $T$ of $\Theta_{S}$ is written 
\[
T=\alpha\partial_{x_{1}}+\beta y_{1}\partial_{y_{1}},\quad\alpha,\beta\in\mathbb{C}\left\{ x_{1},y_{1}\right\} 
\]
and 
\[
\mathfrak{B}\left(T\right)=\left(u^{2}y_{1}\partial_{x_{1}}\beta\right)\dd x_{1}\wedge\dd y_{1}.
\]
\end{proof}
Notice that if $c$ is not a tangency point between $X_{1}^{E}$ and
$D$, then at the level of the stack, one has 
\[
\left(\mathfrak{B}\left(\Theta_{S}\right)\right)_{c}=\left(\Omega^{2}\left(-\overline{n}D-S^{E}\right)\right)_{c},
\]
thus the two sheaves $\mathfrak{B}\left(\Theta_{S}\right)$ and $\Omega^{2}\left(-\overline{n}D-S^{E_{}}\right)$
are essentially equal.

\global\long\def\Hom{\textup{Hom}}%

The proof of the next lemma is a corollary of an adaptation of the
theory of infinitesimal deformations of foliations developped in \cite{Gomez}
by Gòmez-Mont. 
\begin{lem}
\label{lem:vanishing-map} The map in cohomology induced by the inclusion
of Lemma \ref{lemma.inclusion.sheaves}
\[
H^{1}\left(\mathcal{M},\Theta_{S}\right)\xrightarrow{\overline{\mathfrak{B}}_{\mathcal{}}}H^{1}\left(\mathcal{M},\Omega^{2}\left(-\overline{n}D-S^{E}\right)\right)
\]
is the zero map.
\end{lem}
\begin{proof}
Let us denote by $\Theta_{X_{1}}$ the sheaf of tangent vector fields
to the foliation induced on $\mathcal{M}$ by $X_{1}^{E}$$.$ Let
us consider the morphism of sheaves 
\begin{equation}
\Theta_{S}\xrightarrow{\mathfrak{D}}\Hom\left(\Theta_{X_{1}},\Theta_{S}/\Theta_{X_{1}}\right)\label{eq:morphism.sheaf}
\end{equation}
defined by $\mathfrak{D}\left(T\right)=\left(X\mapsto\pi\left[X,T\right]\right)$
where $\left[\cdot\right]$ stands for the Lie bracket and $\pi$
the quotient map $\pi:\Theta_{X_{1}}\to\Theta_{S}/\Theta_{X_{1}}$.
Following \cite{Gomez} (Theorem 1.6), one has the following exact
sequence
\begin{equation}
\mathbb{H}^{1}\left(\mathcal{M},\Theta_{X_{1}}\right)\to H^{1}\left(\mathcal{M},\Theta_{S}\right)\xrightarrow{\overline{\mathfrak{D}}}H^{1}\left(\mathcal{M},\Hom\left(\Theta_{X_{1}},\Theta_{S}/\Theta_{X_{1}}\right)\right).\label{eq:operatorD}
\end{equation}
In this sequence, $\mathbb{H}^{1}\left(\mathcal{M},\Theta_{X_{1}}\right)$
is the first hypercohomology group of the leaf complex associated
to the morphism (\ref{eq:morphism.sheaf}) as defined in \cite{Gomez}.
It is identified with the space of infinitesimal deformations of the
foliation induced by $X_{1}^{E}$. The cohomological group $H^{1}\left(\mathcal{M},\bullet\right)$
is the standard C\v{e}ch cohomology of sheaves. The first cohomology
group $H^{1}\left(\mathcal{M},\Theta_{S}\right)$ is identified with
the space of infinitesimal deformations of $S^{E}$. Finally, $\overline{\mathfrak{D}}$
is the map induced in cohomology by $\mathfrak{D}$.

Assume that $S$ is generic in its moduli space $\moduli.$ Theorem
\ref{Flat.saito.basis} ensures that any small deformation of $S^{E}$
can be followed by a deformation of $X_{1}^{E}.$ As a consequence,
any infinitesimal deformation of $S^{E}$ can be followed by an infinitesimal
deformation of $X_{1}^{E}.$ In other words, in (\ref{eq:operatorD})
the map
\[
\mathbb{H}^{1}\left(\mathcal{M},\Theta_{X_{1}}\right)\to H^{1}\left(\mathcal{M},\Theta_{S}\right)
\]
is onto. Since the sequence (\ref{eq:operatorD}) is exact, the map
$\overline{\mathfrak{D}}$ is the zero map.

Now, let us consider a covering $\left\{ U_{i}\right\} _{i\in I}$
of $\mathcal{M}$ and a cocycle $\left\{ T_{ij}\right\} _{ij}$ 
\[
\left\{ T_{ij}\right\} _{ij}\in Z^{1}\left(\mathcal{M},\left\{ U_{i}\right\} _{i\in I},\Theta_{S}\right).
\]
The map $\overline{\mathfrak{D}}$ being the zero map, the cocycle
$\overline{\mathfrak{D}}\left(\left\{ T_{ij}\right\} _{ij}\right)$
is trivial, that is,
\[
\overline{\mathfrak{D}}\left(\left\{ T_{ij}\right\} _{ij}\right)\equiv0\textup{ in }H^{1}\left(\mathcal{M},\Hom\left(\Theta_{X_{1}},\Theta_{S}/\Theta_{X_{1}}\right)\right).
\]
By definition, there exists $\left\{ \mathcal{T}_{i}\right\} _{i}\in Z^{0}\left(\mathcal{M},\left\{ U_{i}\right\} _{i\in I},\Hom\left(\Theta_{X_{1}},\Theta_{S}/\Theta_{X_{1}}\right)\right)$
such that 
\[
\left[T_{ij},\cdot\right]=\mathcal{T}_{j}-\mathcal{T}_{i}.
\]
At the level of the stack, the map $\mathfrak{D}$ is onto at any
regular point for $X_{1}^{E}$. Thus we can consider a covering of
$U_{i}\setminus\textup{Sing}\left(X_{1}^{E}\right)=\bigcup_{k\in K}U_{ik}$
by open sets $U_{ik}$ such that $\mathfrak{D}$ is onto on $U_{ik}.$
By construction, on any $U_{ik}$ there exists a section $\tau_{ik}$
of $\Theta_{S}$ such that 
\[
\mathcal{T}_{i}=\left[\tau_{ik},\cdot\right].
\]
Therefore, on $U_{ik}\cap U_{ik^{\prime}}$, $\left[\tau_{ik},\cdot\right]=\left[\mathcal{\tau}_{ik^{\prime}},\cdot\right]$.
Thus, appyling $\mathfrak{B}$ yields 
\[
L_{\mathcal{\tau}_{ik}}\Omega\wedge\Omega=\mathfrak{B}\left(\tau_{ik}\right)=\mathfrak{B}\left(\tau_{ik^{\prime}}\right)=L_{\tau_{ik^{\prime}}}\Omega\wedge\Omega.
\]
Therefore, the $2-$forms $\left\{ L_{\mathcal{\tau}_{ik}}\Omega\wedge\Omega\right\} _{k\in K}$
paste in a global $2-$forms $\Omega_{i}$ defined on $U_{i}\setminus\textup{Sing}\left(X_{1}^{E}\right)$
which can be extended to $U_{i}$ since $\textup{Sing}\left(X_{1}^{E}\right)$
is of codimension $2.$ By construction, 
\[
\mathfrak{\overline{B}}\left(\left\{ T_{ij}\right\} \right)\equiv\left\{ \Omega_{j}-\Omega_{i}\right\} ,
\]
which is the lemma.
\end{proof}
The open sets $U_{1}$ and $U_{2}$ defined at the beginning of this
section are Stein as open set in $\mathbb{C}$. Thus, following \cite{SiuThm},
they admit a system of Stein neighborhoods. Since $\Omega^{2}\left(-\overline{n}D-S^{E}\right)$
is coherent, we deduce that there is a covering $\left\{ \mathcal{U}_{1},\mathcal{U}_{2}\right\} $
of $\mathcal{M}$ that is acyclic for $\Omega^{2}\left(-\overline{n}D-S^{E}\right).$
Therefore, one can compute the cohomology using this covering and
thus 
\begin{align*}
H^{1}\left(\mathcal{M},\Omega^{2}\left(-\overline{n}D-S^{E}\right)\right) & =H^{1}\left(\left\{ \mathcal{U}_{1},\mathcal{U}_{2}\right\} ,\Omega^{2}\left(-\overline{n}D-S^{E}\right)\right)
\end{align*}
 which is the quotient 
\begin{equation}
\frac{H^{0}\left(\mathcal{U}_{1}\cap\mathcal{U}_{2},\Omega^{2}\left(-\overline{n}D-S^{E}\right)\right)}{H^{0}\left(\mathcal{U}_{1},\Omega^{2}\left(-\overline{n}D-S^{E}\right)\right)\oplus H^{0}\left(\mathcal{U}_{2},\Omega^{2}\left(-\overline{n}D-S^{E}\right)\right)}.\label{eq:egalite.coho.identification}
\end{equation}

The lemma below is the key to get a lower bound for the Saito number
$\mathfrak{s}\left(S\right)$ of the curve $S.$ 
\begin{lem}
\label{lem:condition-vanish} Let $f_{1}$ be the quotient $\frac{f\circ E}{x_{1}^{\nu\left(S\right)}}$
where $f$ is a reduced equation of $S$. If there exists a Laurent
series $A=\sum a_{i,j}x_{1}^{i}y_{1}^{j}$ holomorphic on $\mathcal{U}_{1}\cap\mathcal{U}_{2}$
with a non vanishing residu $a_{0,-1},$ such that, in the identification
(\ref{eq:egalite.coho.identification}), one has

\[
\left[A\cdot f_{1}x_{1}^{k}\dd x_{1}\wedge\dd y_{1}\right]\equiv0\in H^{1}\left(\mathcal{M},\Omega^{2}\left(-kD-S^{E}\right)\right)
\]

then 
\[
k\geq\nu\left(S\right).
\]
\end{lem}
\begin{proof}
The global sections of $\Omega^{2}\left(-kD-S^{E}\right)$ on each
open sets $\mathcal{U}_{1}$, $\mathcal{U}_{2}$ and their intersection
are written 
\begin{eqnarray*}
\Omega^{2}\left(-kD-S^{E}\right)\left(\mathcal{U}_{1}\right) & = & \left\{ \left.f\left(x_{1},y_{1}\right)f_{1}x_{1}^{k}\dd x_{1}\wedge\dd y_{1}\right|f\in\mathcal{O}\left(\mathcal{U}_{1}\right)\right\} \\
\Omega^{2}\left(-kD-S^{E}\right)\left(\mathcal{U}_{2}\right) & = & \left\{ \left.g\left(x_{2},y_{2}\right)f_{2}y_{2}^{k}\dd x_{2}\wedge\dd y_{2}\right|g\in\mathcal{O}\left(\mathcal{U}_{2}\right)\right\} \\
\Omega^{2}\left(-kD-S^{E}\right)\left(\mathcal{U}_{1}\cap\mathcal{U}_{2}\right) & = & \left\{ \left.h\left(x_{1},y_{1}\right)f_{1}x_{1}^{k}\dd x_{1}\wedge\dd y_{1}\right|h\in\mathcal{O}\left(\mathcal{U}_{1}\cap\mathcal{U}_{2}\right)\right\} 
\end{eqnarray*}
where $f_{2}=\frac{f\circ E}{y_{2}^{\nu\left(S\right)}}$. Therefore,
the cohomological equation induced by the equality (\ref{eq:egalite.coho.identification})
is written
\[
\begin{array}{r}
h\left(x_{1},y_{1}\right)f_{1}x_{1}^{k}\dd x_{1}\wedge\dd y_{1}=g\left(x_{2},y_{2}\right)f_{2}y_{2}^{k}\dd x_{2}\wedge\dd y_{2}\\
\quad-f\left(x_{1},y_{1}\right)f_{1}x_{1}^{k}\dd x_{1}\wedge\dd y_{1}
\end{array}
\]
which is equivalent to
\begin{equation}
h\left(x_{1},y_{1}\right)=y_{1}^{k-\nu\left(S\right)-1}g\left(\frac{1}{y_{1}},y_{1}x_{1}\right)-f\left(x_{1},y_{1}\right)\label{eq:coho}
\end{equation}
The hypothesis of Lemma \ref{lem:condition-vanish} induces that if
we set $h$ to be the series $\sum a_{i,j}x_{1}^{i}y_{1}^{j}$ then
the equation above has a solution. In particular, the monomial $\frac{a_{0,-1}}{y_{1}}$
has to appear in the Laurent expansion of one of the two terms of
the expression at the right of (\ref{eq:coho}). This is equivalent
to require that the following system and 
\[
\begin{cases}
0=j\\
-1=j-i+k-\nu\left(S\right)-1
\end{cases}\Longleftrightarrow\begin{cases}
j=0\\
i=k-\nu\left(S\right)
\end{cases}.
\]
has a solution in $\mathbb{N}^{2}$. Thus, $k\geq\nu\left(S\right)$.
\end{proof}
\begin{thm}
\label{thm:LowerGenericBound}For $S$ generic in its moduli space
$\moduli$, one has 
\[
\mathfrak{s}\left(S\right)\geq\left\{ \begin{array}{ll}
\left\lfloor \frac{\nu\left(S\right)}{2}\right\rfloor  & \textup{ if }S\textup{ is not of radial type}\\
\\
\left\lceil \frac{\nu\left(S\right)}{2}\right\rceil -1 & \textup{else}
\end{array}\right.,
\]
where $\left\lfloor \star\right\rfloor $ and $\left\lceil \star\right\rceil $
stands respectively for the integer part and the least integer of
$\star$. 
\end{thm}
In the moduli space $\moduli,$ the lower bound above holds only for
the generic point. For instance, the Saito number of a union of any
number of germs of straight lines is 1, since the radial vector field
$x\partial x+y\partial y$ is in the Saito module, whereas the algebraic
multiplicity $\nu\left(S\right)$ goes to infinity with the number
of components. Even if the curve $S$ is irreducible, one cannot drop
the assumption of $S$ being generic in its moduli space, as it can
be seen in the following example due to M. Hernandes known as \emph{deformation
by socle} : let $S$ be the irreducible curve 
\[
\left\{ y^{p}-x^{q}+x^{q-2}y^{p-2}=0\right\} 
\]
with $p\wedge q=1$ and $4=p<q.$ Its algebraic multiplicity is equal
to $p$ whereas its Saito number $\mathfrak{s}\left(S\right)$ is
equal to $2$ regardless the value of $p$. Indeed, the vector field
$X_{1}$ written
\begin{align*}
X_{1} & =\left(y+\frac{\left(p-2\right)\left(q-2\right)}{pq}x^{q-4}y^{p-3}\right)\left(px\partial_{x}+qy\partial_{y}\right)\\
 & +\frac{\left(p-2\right)q-2p}{q}x^{q-2}\partial_{y}-\left(p-2\right)\frac{\left(p-2\right)q-2p}{pq}x^{p-3}y^{q-3}\partial_{x}
\end{align*}
is optimal for $S.$ 
\begin{proof}[Proof of Theorem \ref{thm:LowerGenericBound}]
Let $X_{1}$ be a generic optimal vector field for $S.$ Since we
assume $S$ generic in its moduli space, the operator $\overline{\mathfrak{B}}$
associated to $X_{1}$ and defined in Lemma \ref{lem:vanishing-map}
is trivial.

Suppose, first that $X_{1}$ is dicritical. Let us suppose that in
the coordinates $\left(x_{1},y_{1}\right)$, the vector field $X_{1}^{E}$
is transverse to $D$ at $\left(0,0\right)$ and that $f_{1}=\frac{f\circ E}{x_{1}^{\nu\left(S\right)}}$
does not vanish at $\left(0,0\right)$. We can suppose that, in these
coordinates, $\Omega$ is written
\[
\Omega=ux_{1}^{\nu\left(X_{1}\right)+1}\dd y_{1},\quad u\left(0\right)\neq0.
\]
The image of the vector field 
\[
T=\frac{x_{1}}{y_{1}}\partial_{y_{1}}
\]
 by $\frac{1}{f_{1}}\mathfrak{B}$ is written 
\[
\frac{1}{f_{1}}\mathfrak{B}\left(T\right)=\frac{1}{f_{1}}L_{T}\Omega\wedge\Omega=\frac{u^{2}\left(0,0\right)}{f_{1}\left(0,0\right)}x_{1}^{\overline{n}}\frac{1}{y_{1}}\dd x_{1}\wedge\dd y_{1}+x_{1}^{\overline{n}+1}\left(\cdots\right).
\]
This meromorphic $2-$form considered as a cocycle in $Z^{1}\left(\left\{ \mathcal{U}_{1},\mathcal{U}_{2}\right\} ,\Omega^{2}\left(-\overline{n}D-S^{E}\right)\right)$
has to be trivial in cohomology according to Lemma \ref{lem:vanishing-map}.
Thus, Lemma \ref{lem:condition-vanish} ensures that $\overline{n}=2\nu\left(X_{1}\right)+2\geq\nu\left(S\right),$
which is also written 
\[
\mathfrak{s}\left(S\right)=\nu\left(X_{1}\right)\geq\frac{\nu\left(S\right)}{2}-1.
\]
Therefore, if $X_{1}$ is dicritical the theorem is proved.

Suppose now that $X_{1}$ is not dicritical. Let us suppose that $\left(0,0\right)$
is a singular point of $X_{1}^{E}$. Locally around $\left(0,0\right)$,
$\Omega$ can be written 
\[
\Omega=x_{1}^{\nu\left(X_{1}\right)}y_{1}^{a}\dd x_{1}+x_{1}^{\nu\left(X_{1}\right)+1}\left(\cdots\right)
\]
where $a$ is some positive integer. Let us write
\[
f_{1}=y_{1}^{b}v\left(y_{1}\right)+x_{1}\left(\cdots\right),\quad v\left(0\right)\neq0
\]
where $b$ is some positive integer. Considering the meromorphic vector
field 
\[
T=\frac{x_{1}}{y_{1}^{2a-b}}\partial_{x_{1}},
\]
we apply the operator $\frac{1}{f_{1}}\mathfrak{B}$ and obtain 
\[
\frac{1}{f_{1}}\mathfrak{B}\left(T\right)=\frac{\left(2a-b\right)}{v\left(0\right)}\frac{x_{1}^{\overline{n}}}{y_{1}}\dd x_{1}\wedge\dd y_{1}+x_{1}^{\overline{n}+1}\left(\cdots\right).
\]
Suppose that there exists a singular point of $X_{1}^{E}$ such that
$2a\neq b.$ Then, Lemma \ref{lem:condition-vanish} ensures that
$\overline{n}=2\nu\left(X_{1}\right)+1\geq\nu\left(S\right),$ which
is written 
\begin{equation}
\mathfrak{s}\left(S\right)=\nu\left(X_{1}\right)\geq\frac{\nu\left(S\right)-1}{2}.\label{inegalite.bien}
\end{equation}
If the equality $2a=b$ is true for any singular points, then $\nu\left(S\right)$
is even. Thus, the theorem is proved when
\begin{itemize}
\item $\nu\left(S\right)$ is odd
\item or $\nu\left(S\right)$ is even and for some singular points of $X_{1}^{E}$,
one has $b\neq2a.$
\item or if $S$ is radial.
\end{itemize}
Finally, suppose that $\nu\left(S\right)$ is even and $S$ is not
radial. Consider a Saito basis $\left\{ X_{1},X_{2}\right\} $ for
$S$ with $\nu\left(X_{1}\right)=\nu\left(X_{2}\right)$. If $\nu\left(X_{1}\right)=\frac{\nu\left(S\right)}{2}$
then the property is proved. Therefore, assume that $\nu\left(X_{1}\right)\leq\frac{\nu\left(S\right)}{2}-1$.
Let $l_{1}$ be a generic smooth curve. Using the construction introduced
at (\ref{ajouter.une.courbe}), we obtain a Saito basis for $S\cup l_{1}$
of the form 
\[
\left\{ X_{1}+\phi_{1}X_{2},L_{1}X_{2}\right\} ,\ \phi_{1}\left(0\right)\neq0\textup{ and }l_{1}=\left\{ L_{1}=0\right\} 
\]
If $\nu\left(X_{1}+\phi_{1}X_{2}\right)=\nu\left(X_{1}\right)$ then
\[
\nu\left(X_{1}+\phi_{1}X_{2}\right)\leq\frac{\nu\left(S\right)}{2}-1<\frac{\nu\left(S\cup l_{1}\right)-1}{2}
\]
which contradicts Theorem \ref{thm:LowerGenericBound} applied to
$S\cup l_{1}$, the valuation $\nu\left(S\cup l_{1}\right)$ being
odd. Therefore, $\nu\left(X_{1}+\phi_{1}X_{2}\right)\geq\frac{\nu\left(S\right)}{2}$
and since $\nu\left(L_{1}X_{2}\right)\geq\frac{\nu\left(S\right)}{2}$
and $\nu\left(X_{2}\right)\leq\frac{\nu\left(s\right)}{2}-1$, considering
if necessary $X_{1}+\phi_{1}X_{2}+L_{1}X_{2},$ we obtain a basis
of Saito for $S\cup l_{1}$ written
\begin{equation}
\left\{ X_{1}+\phi_{1}X_{2},L_{1}X_{2}\right\} ,\qquad\phi_{1}\left(0\right)\ne0\label{eq:premiere.base}
\end{equation}
both of these vector fields being non dicritical and of multiplicity
$\frac{\nu\left(S\right)}{2}$. Using again the construction (\ref{ajouter.une.courbe}),
we add one more generic curve $l_{2}$ and obtain a basis of Saito
of the form
\[
\left\{ \underbrace{L_{2}\left(X_{1}+\phi_{1}X_{2}\right)}_{Y_{1}},\underbrace{L_{1}X_{2}+\phi_{2}\left(X_{1}+\phi_{1}X_{2}\right)}_{Y_{2}}\right\} ,~\phi_{2}\left(0\right)\neq0\textup{ and }l_{2}=\left\{ L_{2}=0\right\} .
\]
We can apply Theorem \ref{thm:LowerGenericBound} to $S\cup l_{1}\cup l_{2}$
since the latter curve has a smooth component for which $b=1$ is
not even. Therefore, the two above vector fields are of multiplicity
$\frac{\nu}{2}+1$ and not dicritical. According to the Saito criterion
applied to (\ref{eq:premiere.base}), one has
\[
\left(X_{1}+\phi_{1}X_{2}\right)\wedge L_{1}X_{2}=uL_{1}f,\ u\left(0\right)\neq0.
\]
Therefore, $L_{1}$ cannot divide $X_{1}+\phi_{1}X_{2}.$ Now consider
any couple of non vanishing functions $\alpha$ and $\beta$. Writing
\begin{equation}
\alpha Y_{1}+\beta Y_{2}=\left(\beta\phi_{2}+\alpha L_{2}\right)\left(X_{1}+\phi_{1}X_{2}\right)+\beta L_{1}X_{2}\label{eq:crit.div}
\end{equation}
ensures that $\alpha Y_{1}+\beta Y_{2}$ cannot be divided by $L_{1}L_{2}$.
Fix some coordinates $\left(x,y\right)$ such that $L_{1}=x$ and
$L_{2}=y$. Taking a suitable linear combination of $Y_{1}$ and $Y_{2}$
we can suppose that they are written 
\begin{align*}
Y_{1} & =a\left(x\right)x^{p}\partial_{x}+b\left(y\right)y^{q}\partial_{y}+xy\left(\cdots\right)\\
Y_{2} & =c\left(x\right)x^{p}\partial_{x}+d\left(y\right)y^{q}\partial_{y}+xy\left(\cdots\right)
\end{align*}
where $a,b,c$ and $d$ are non-vanishing germs of functions and $p$
and $q$ some integers bigger than $\frac{\nu}{2}+1$. Dividing $Y_{1}$
and $Y_{2}$ respectively by $b$ and $d$, and making a suitable
change of coordinates of the form $\left(x,y\right)\mapsto\left(u\left(x\right),y\right)$,
we can suppose that $Y_{1}$ and $Y_{2}$ are written 
\begin{align*}
Y_{1} & =ax^{p}\partial_{x}+by^{q}\partial_{y}+xy\left(\cdots\right)\\
Y_{2} & =c\left(x\right)x^{p}\partial_{x}+dy^{q}\partial_{y}+xy\left(\cdots\right)
\end{align*}
where $a,b$ and $d$ belongs to $\mathbb{C}\setminus\left\{ 0\right\} $.
Finally, considering the vector field 
\[
Y_{2}-\frac{\left(c\left(x\right)-c\left(0\right)\right)}{a}Y_{1},
\]
we can write
\begin{align*}
Y_{1} & =ax^{p}\partial_{x}+by^{q}\partial_{y}+xy\left(\cdots\right)\\
Y_{2} & =cx^{p}\partial_{x}+dy^{q}\partial_{y}+xy\left(\cdots\right)
\end{align*}
where $a,b,c$ and $d$ are non vanishing complex numbers. Now, the
Saito criterion written
\[
Y_{1}\wedge Y_{2}=uxyf
\]
where $u$ is a unit ensures that
\[
\left(cY_{1}-aY_{2}\right)\wedge\left(dY_{1}-bY_{2}\right)=\left(ad-bc\right)Y_{1}\wedge Y_{2}=u\left(ad-bc\right)xyf.
\]
If $ad-bc=0$ then considering $\left(\begin{array}{c}
\alpha\\
\beta
\end{array}\right)$ in the kernel of the matrix $\left(\begin{array}{cc}
a & c\\
b & d
\end{array}\right)$ yields a linear combination written
\begin{equation}
\alpha Y_{1}+\beta Y_{2}=xy\left(\cdots\right)\label{eq:division}
\end{equation}
Since neither $Y_{1}$ nor $Y_{2}$ can be divided by $xy$, one has
$\alpha\neq0$ and $\beta\neq0.$ According to (\ref{eq:crit.div}),
$\alpha Y_{1}+\beta Y_{2}$ cannot be divided by $xy$ too. That is
a contradiction with (\ref{eq:division}). Therefore, $ad-bc\neq0$
and the expression 
\[
\frac{\left(cY_{1}-aY_{2}\right)}{y}\wedge\frac{\left(dY_{1}-bY_{2}\right)}{x}=u\left(ad-bc\right)f
\]
is the Saito criterion for the curve $S.$ However, both vector fields
in the product above have mutiplicities bigger than $\frac{\nu\left(S\right)}{2}$
which is a contradiction with the initial assumption.
\end{proof}

\subsection{Generic Saito basis.}

The generic lower bound of Theorem \ref{thm:LowerGenericBound} induces
some properties for a Saito basis of a generic curve. In this section,
we explore some of them. 

To do so, we are going to use frequently the following lemma that
is a direct consequence of the criterion of Saito.
\begin{lem}
\label{lem:0or-1}Let $\left\{ X_{1},X_{2}\right\} $ be a Saito basis
$S.$ Then
\begin{enumerate}
\item if $\nu\left(X_{1}\right)+\nu\left(X_{2}\right)<\nu\left(S\right)$
and $X_{1}$ is dicritical then $X_{2}$ is dicritical.
\item if $\nu\left(X_{1}\right)+\nu\left(X_{2}\right)=\nu\left(S\right)$
and $X_{1}$ is dicritical then $X_{2}$ is not dicritical.
\item If $S$ is generic
\item in its moduli space, then one can suppose that 
\[
\nu\left(S\right)-1\leq\nu\left(X_{1}\right)+\nu\left(X_{2}\right)\leq\nu\left(S\right).
\]
\end{enumerate}
\end{lem}
\begin{proof}
Properties $\left(1\right)$ and $\left(2\right)$ are consequences
of the following remark
\[
X_{1}^{\left(\nu\left(X_{1}\right)\right)}\wedge X_{2}^{\left(\nu\left(X_{2}\right)\right)}\equiv0\Longleftrightarrow\nu\left(X_{1}\right)+\nu\left(X_{2}\right)<\nu\left(S\right).
\]
For $\left(3\right),$ first, we recall that the sum $\nu\left(X_{1}\right)+\nu\left(X_{2}\right)$
cannot exceed $\nu\left(S\right)$ as noticed in (\ref{inegalite.fondamentale.saito}).
If $\nu\left(S\right)$ is odd, Theorem \ref{thm:LowerGenericBound}
gives the inequalities 
\[
\nu\left(X_{i}\right)\geq\frac{\nu\left(S\right)-1}{2},\ i=1,2
\]
and thus $\nu\left(X_{1}\right)+\nu\left(X_{2}\right)\ge\nu\left(S\right)-1$,
which is the lemma.

If $\nu\left(S\right)$ is even, adding a generic line $l$ to $S$
yields a Saito basis of $S\cup l$ for which, in view of the previous
arguments - $\nu\left(S\cup l\right)$ is odd - , one has
\[
\nu\left(X_{1}\right)+\nu\left(X_{2}\right)\geq\nu\left(S\cup l\right)-1=\nu\left(S\right).
\]
By the process described in Proposition \ref{division.saito.process},
the induced Saito basis $\left\{ X_{1}^{\prime},X_{2}\right\} $ of
$S$ satisfies
\[
\nu\left(X_{1}^{\prime}\right)+\nu\left(X_{2}\right)\geq\nu\left(X_{1}\right)+\nu\left(X_{2}\right)-1\geq\nu\left(S\right)-1,
\]
which ends the proof of the lemma.
\end{proof}
The next lemma ensures somehow that both inequalities identified in
Theorem \ref{thm:LowerGenericBound} cannot be reached at the same
time.
\begin{lem}
\label{lemma-cannot-happen-both}Let $S$ be a generic curve of radial
type. Then there is no non dicritical vector field $X$ in $\textup{Der}\left(\log S\right)$
with $\nu\left(X\right)=\left\lfloor \frac{\nu\left(S\right)}{2}\right\rfloor $.
\end{lem}
\begin{proof}
Consider an optimal dicritical vector field $X_{1}$ and $X_{2}$
a vector field such that $\left\{ X_{1},X_{2}\right\} $ is a Saito
basis of S. If $\nu\left(S\right)$ is even, then Theorem \ref{eq:operatorD}
ensures that $\nu\left(X_{1}\right)\geq\frac{\nu\left(S\right)}{2}-1$.
If $\nu\left(X_{1}\right)=\frac{\nu\left(S\right)}{2}$ then the lemma
follows from the definition of $S$ being radial. If $\nu\left(X_{1}\right)=\frac{\nu\left(S\right)}{2}-1$
then either $X_{1}^{\left(\nu\left(X_{1}\right)\right)}\wedge X_{2}^{\left(\nu\left(X_{2}\right)\right)}=0$
or $\nu\left(X_{2}\right)\geq\frac{\nu\left(S\right)}{2}+1.$ In any
case, the lemma follows. Finally, if $\nu\left(S\right)$ is odd and
$S$ radial, by definition, every vector field of multiplicity $\left\lfloor \frac{\nu\left(S\right)}{2}\right\rfloor =\frac{\nu\left(S\right)-1}{2}$
is dicritical.
\end{proof}
In the proposition below, we are going to identify precisely the type
of Saito basis that may occur for a generic curve. In the statement
of the theorem, we introduce some notations for the identified classes.
\begin{thm}
\label{SaitoBasisForm}Let $S$ be a curve generic in its moduli space.
Then there exists a Saito basis $\left\{ X_{1},X_{2}\right\} $ for
$S$ with one of the following forms
\begin{itemize}
\item if $\nu\left(S\right)$ is even
\begin{itemize}
\item[$\left(\mathfrak{E}\right)$] : $\nu\left(X_{1}\right)=\nu\left(X_{2}\right)=\frac{\nu\left(S\right)}{2}$,
$X_{1}$ and $X_{2}$ are non dicritical.
\item[$\left(\mathfrak{E}_{d}\right)$] : $\nu\left(X_{1}\right)=\nu\left(X_{2}\right)-1=\frac{\nu\left(S\right)}{2}-1$,
$X_{1}$ and $X_{2}$ are dicritical.
\item[$\left(\mathfrak{E}_{d}^{\prime}\right)$] : $\nu\left(X_{1}\right)=\nu\left(X_{2}\right)-2=\frac{\nu\left(S\right)}{2}-1$,
$X_{1}$ is dicritical but not $X_{2}.$
\end{itemize}
\item if $\nu\left(S\right)$ is odd
\begin{itemize}
\item[$\left(\mathfrak{O}\right)$] : $\nu\left(X_{1}\right)=\nu\left(X_{2}\right)-1=\frac{\nu-1}{2}$,
$X_{1}$ and $X_{2}$ are non dicritical.
\item[$\left(\mathfrak{O}_{d}\right)$] : $\nu\left(X_{1}\right)=\nu\left(X_{2}\right)=\frac{\nu-1}{2}$,
$X_{1}$ and $X_{2}$ are dicritical.
\item[$\left(\mathfrak{O}_{d}^{\prime}\right)$] : $\nu\left(X_{1}\right)=\nu\left(X_{2}\right)-1=\frac{\nu-1}{2}$,
$X_{1}$ is dicritical but not $X_{2}.$
\end{itemize}
\end{itemize}
Moreover, if $\left\{ X_{1},X_{2}\right\} $ is a generic Saito basis
for $S$ then there exists an holomorphic function $h$ such that
\[
\left\{ X_{1},X_{2}-hX_{1}\right\} 
\]
 has one of the above type.
\end{thm}
If the Saito basis of $S$ has one of the forms given by Theorem \ref{SaitoBasisForm},
we will say that the basis is \emph{adapted.}
\begin{rem}
Notice that if the Saito basis $\left\{ X_{1},X_{2}\right\} $ of
$S$ is of type $\left(\mathfrak{E}_{d}^{\prime}\right)$ or $\left(\mathfrak{O}_{d}^{\prime}\right)$
then for any function $L$ with a non trivial linear part, the family
\[
\left\{ X_{1},X_{2}+LX_{1}\right\} ,\quad\left\{ \begin{array}{c}
L\left(0\right)\neq0\textup{ if \ensuremath{S} is of type \ensuremath{\left(\mathfrak{O}_{d}^{\prime}\right)}}\\
L\left(0\right)=0\textup{ if \ensuremath{S} is of type \ensuremath{\left(\mathfrak{E}_{d}^{\prime}\right)}}
\end{array}\right.
\]
 is a Saito basis for $S$ of type $\left(\mathfrak{E}_{d}\right)$
or $\left(\mathfrak{O}_{d}\right)$. In some sense, the bases of type
$\left(\mathfrak{E}_{d}^{\prime}\right)$ or $\left(\mathfrak{O}_{d}^{\prime}\right)$
are exceptional among the one of type $\left(\mathfrak{E}_{d}\right)$
or $\left(\mathfrak{O}_{d}\right)$.
\end{rem}
\begin{rem}
\label{rem:EDprime}The curves of type $\left(\mathfrak{E}_{d}^{\prime}\right)$
are the only curves for which there exists a Saito basis $\left\{ X_{1},X_{2}\right\} $
with 
\[
\left|\nu\left(X_{1}\right)-\nu\left(X_{2}\right)\right|\geq2
\]
\end{rem}
\begin{rem}
\textcolor{red}{\label{relevement.base.adaptee} }\textcolor{black}{One
of the interest of adapted Saito bases is their behaviour with respect
to the blowing-up. For instance, suppose that $S$ has an adapted
Saito basis $\left\{ X_{1},X_{2}\right\} $ of type $\left(\mathfrak{E}_{d}\right).$
Then, blowing-up the Saito criterion (\ref{Saito.fondamental.wedge})
yields the relation
\[
X_{1}^{E}\wedge X_{2}^{E}=u\circ E\frac{f\circ E}{x_{1}^{\nu\left(S\right)}}.
\]
Therefore, according to the Saito criterion, the family $\left\{ \left(X_{1}^{E}\right)_{c},\left(X_{2}^{E}\right)_{c}\right\} $
is a Saito basis for $\left(S^{E}\right)_{c}$ for any $c\in D$ -
but not necessarly }\textcolor{black}{\emph{adapted}}\textcolor{black}{.
It is a simple matter to check that the latter proprety holds for
any above type of Saito bases.}
\end{rem}
\begin{proof}[Proof of Theorem \ref{SaitoBasisForm}]
Let us consider a Saito basis $\left\{ X_{1},X_{2}\right\} $ of
$S$ and suppose that $\nu\left(X_{1}\right)\leq\nu\left(X_{2}\right).$
\begin{casenv}
\item Suppose first $\nu\left(S\right)$ even. If $X_{1}$ is not dicritical
then according to Theorem \ref{thm:LowerGenericBound} and (\ref{inegalite.fondamentale.saito}),
$\nu\left(X_{1}\right)=\nu\left(X_{2}\right)=\frac{\nu\left(S\right)}{2}$.
Considering if necessary $X_{2}+\alpha X_{1}$ for a generic value
$\alpha\in\mathbb{C}$, one has 
\[
\left(\mathfrak{E}\right)\qquad\begin{array}{l}
\nu\left(X_{1}\right)=\nu\left(X_{2}\right)=\frac{\nu\left(S\right)}{2}\\
X_{1}\text{ and }X_{2}\textup{ are non-dicritical.}
\end{array}
\]
Assume $X_{1}$ is dicritical. If $\nu\left(X_{1}\right)=\frac{\nu\left(S\right)}{2}$
then Lemma \ref{lem:0or-1} ensures that $X_{2}$ is not dicritical
and the Saito basis $\left\{ X_{1}+X_{2},X_{2}\right\} $ is of type
$\left(\mathfrak{E}\right).$ If $\nu\left(X_{1}\right)=\frac{\nu\left(S\right)}{2}-1$,
following Lemma \ref{lem:0or-1}, one can suppose that
\[
\nu\left(X_{2}\right)=\frac{\nu\left(S\right)}{2}\textup{ or }\frac{\nu\left(S\right)}{2}+1.
\]
If $\nu\left(X_{2}\right)=\frac{\nu\left(S\right)}{2}+1$ then $X_{2}$
is not dicritical. Thus, one has a basis of the form
\[
\left(\mathfrak{E}_{d}^{\prime}\right)\qquad\begin{array}{l}
\nu\left(X_{1}\right)=\nu\left(X_{2}\right)-2=\frac{\nu\left(S\right)}{2}-1\\
X_{1}\text{ is dicritical but not }X_{2}.
\end{array}
\]
If $\nu\left(X_{2}\right)=\frac{\nu\left(S\right)}{2}$, then $X_{2}$
is dicritical, and thus 
\[
\left(\mathfrak{E}_{d}\right)\qquad\begin{array}{l}
\nu\left(X_{1}\right)=\nu\left(X_{2}\right)-1=\frac{\nu\left(S\right)}{2}-1\\
X_{1}\text{ and }X_{2}\textup{ are both dicritical.}
\end{array}
\]
\item Suppose now $\nu\left(S\right)$ odd. In any case, $\nu\left(X_{1}\right)=\frac{\nu\left(S\right)-1}{2}.$
Suppose $X_{1}$ dicritical. If $\nu\left(X_{2}\right)=\frac{\nu\left(S\right)-1}{2}$
then $X_{2}$ is dicritical, and thus 
\[
\left(\mathfrak{O}_{d}\right)\qquad\begin{array}{l}
\nu\left(X_{1}\right)=\nu\left(X_{2}\right)=\frac{\nu\left(S\right)-1}{2}\\
X_{1}\text{ and }X_{2}\textup{ are dicritical.}
\end{array}
\]
If $\nu\left(X_{2}\right)=\frac{\nu\left(S\right)+1}{2}$ then $X_{2}$
is not dicritical, and therefore the basis satifies
\[
\left(\mathfrak{O}_{d}^{\prime}\right)\qquad\begin{array}{l}
\nu\left(X_{1}\right)=\nu\left(X_{2}\right)-1=\frac{\nu\left(S\right)-1}{2}\\
X_{1}\text{ is dicritical and }X_{2}\textup{ is non-dicritical.}
\end{array}
\]
Finally, suppose that $X_{1}$ is not dicritical. If $\nu\left(X_{2}\right)=\frac{\nu\left(S\right)+1}{2}$
then the basis satifies 
\[
\left(\mathfrak{O}\right)\qquad\begin{array}{l}
\nu\left(X_{1}\right)=\nu\left(X_{2}\right)-1=\frac{\nu\left(S\right)-1}{2}\\
X_{1}\text{ and }X_{2}\textup{ are non-dicritical.}
\end{array}
\]
It remains to study the case in which $X_{1}$ is not dicritical and
$\nu\left(X_{2}\right)=\frac{\nu\left(S\right)-1}{2}$. To do so,
consider a generic line $l$. The multiplicity of $S\cup l$ is even,
thus we can apply the results above to reach the description of the
possible bases for $S$.
\begin{enumerate}
\item Suppose first that the Saito basis $\left\{ X_{1}^{l},X_{2}^{l}\right\} $
of $S\cup l$ has the form $\left(\mathfrak{E}\right)$ ;
\[
\nu\left(X_{1}^{l}\right)=\nu\left(X_{2}^{l}\right)=\frac{\nu\left(S\right)+1}{2},
\]
none of these vector fields being dicritical. Let us consider some
coordinates in which $l=\left\{ x=0\right\} $ and let us write
\[
X_{i}^{l}=xA_{i}\partial_{x}+\left(y^{\alpha_{i}}b_{i}\left(y\right)+xB_{i}\right)\partial_{y}
\]
with $b_{i}\left(0\right)\neq0.$ By symmetry, one can suppose $\alpha_{1}\leq\alpha_{2}.$
Thus, the family 
\[
\left\{ X_{1}^{l},\overline{X}_{2}^{l}=\frac{1}{x}\left(X_{2}-y^{\alpha_{2}-\alpha_{1}}\frac{b_{2}}{b_{1}}X_{1}\right)\right\} 
\]
is a Saito basis for $S$ such that 
\[
\nu\left(X_{1}^{l}\right)=\frac{\nu\left(S\right)+1}{2}\textup{ and }\nu\left(\overline{X_{2}}^{l}\right)=\frac{\nu\left(S\right)-1}{2}.
\]
$\overline{X_{2}}^{l}$ has to be not dicritical since $X_{1}$ is
not dicritical. Therefore, $S$ admits a basis of the form $\left(\mathfrak{O}\right).$
\item Suppose that the Saito basis $\left\{ X_{1}^{l},X_{2}^{l}\right\} $
of $S\cup l$ has the form $\left(\mathfrak{E}_{d}\right)$
\[
\nu\left(X_{1}^{l}\right)=\nu\left(X_{2}^{l}\right)-1=\frac{\nu\left(S\right)+1}{2}-1
\]
both vector fields being dicritical. As before, let us consider some
coordinates in which $l=\left\{ x=0\right\} $ and let us written
\[
X_{i}^{l}=xA_{i}\partial_{x}+\left(y^{\alpha_{i}}b_{i}\left(y\right)+xB_{i}\right)\partial_{y}
\]
with $b_{i}\left(0\right)\neq0.$ If $\alpha_{1}\leq\alpha_{2}$ then
the induced Saito basis $\left\{ X_{1}^{l},\overline{X}_{2}^{l}\right\} $
for $S$ satisfies $\nu\left(X_{1}^{l}\right)=\frac{\nu\left(S\right)-1}{2}.$
Therefore, 
\[
\nu\left(\overline{X_{2}}^{l}\right)=\frac{\nu\left(S\right)-1}{2}\textup{ or }\frac{\nu\left(S\right)+1}{2}.
\]
In any case of the alternative above, there is no non dicritical vector
fields of multiplicity $\frac{\nu\left(S\right)-1}{2}$ in the Saito
module of $S,$ which is a contradiction with the property of $X_{1}.$
Therefore, $\alpha_{1}>\alpha_{2}$ . In the induced basis $\left\{ \overline{X}_{1}^{l},X_{2}^{l}\right\} $,
the vector field $\overline{X}_{1}^{l}$ is write
\[
\overline{X}_{1}^{l}=\frac{1}{x}\left(X_{1}^{l}-y^{\alpha_{1}-\alpha_{2}}\frac{b_{1}\left(y\right)}{b_{2}\left(y\right)}X_{2}^{l}\right).
\]
Therefore, $\overline{X}_{1}^{l}$ is dicritical since $\nu\left(X_{2}^{l}\right)>\nu\left(X_{1}^{l}\right).$
But since $X_{1}$ is not dicritical, it is a contradiction.
\item Finally, suppose that the Saito basis of $S\cup l$ has the form $\left(\mathfrak{E}_{d}^{\prime}\right)$
\[
\nu\left(X_{1}^{l}\right)=\nu\left(X_{2}^{l}\right)-2=\frac{\nu\left(S\right)+1}{2}-1
\]
with $X_{1}^{l}$ dicritical and $X_{2}^{l}$ not dicritical. Then
for any linear function $L$ the Saito basis for $S\cup l$ 
\[
\left\{ X_{1}^{l},X_{2}^{l}+LX_{1}^{l}\right\} 
\]
is of type $\left(\mathfrak{E}_{d}\right)$ and we are reduced to
the previous case.
\end{enumerate}
\end{casenv}
\end{proof}
\begin{table}[h]
\begin{tabular}{ccccc}
\addlinespace
{\small{}$S$} & {\scriptsize{}$f=x$} & {\scriptsize{}$f=xy$} & {\scriptsize{}$f=xy\left(x+y\right)$} & {\scriptsize{}$f=xy\left(x^{2}-y^{2}\right)$}\tabularnewline\addlinespace
\midrule
\addlinespace
{\small{}$\nu\left(S\right)$} & {\scriptsize{}$1$} & {\scriptsize{}$2$} & {\scriptsize{}$3$} & {\scriptsize{}$4$}\tabularnewline\addlinespace
\midrule
\addlinespace
{\small{}$X_{1},\ X_{2}$} & {\scriptsize{}$\partial_{x},\ x\partial_{y}$} & {\scriptsize{}$x\partial_{x},\ y\partial_{y}$} & {\scriptsize{}$x\partial_{x}+y\partial_{y},~\sharp f$} & {\scriptsize{}$x\partial_{x}+y\partial_{y},~\sharp f$}\tabularnewline\addlinespace
\midrule
\addlinespace
{\small{}$\nu\left(X_{1}\right),\ \nu\left(X_{2}\right)$} & {\scriptsize{}$0,\ 1$} & {\scriptsize{}$1,\ 1$} & {\scriptsize{}$1,\ 2$} & {\scriptsize{}$1,\ 3$}\tabularnewline\addlinespace
\midrule
\addlinespace
{\small{}Type} & {\scriptsize{}$\left(\mathfrak{O}\right)$} & {\scriptsize{}$\left(\mathfrak{E}\right)$} & {\scriptsize{}$\left(\mathfrak{O}_{d}^{\prime}\right)$} & {\scriptsize{}$\left(\mathfrak{E}_{d}^{\prime}\right)$}\tabularnewline\addlinespace
\end{tabular}

\begin{tabular}{ccc}
\noalign{\vskip\doublerulesep}
{\small{}$S$} & {\scriptsize{}$f=xy\left(x^{3}-y^{3}+\cdots\right)$} & {\scriptsize{}$f=xy\left(x^{2}-y^{2}\right)\left(x+2y+\cdots\right)\left(x+3y+\cdots\right)$}\tabularnewline[\doublerulesep]
\hline 
\noalign{\vskip\doublerulesep}
{\small{}$\nu\left(S\right)$} & {\scriptsize{}5} & {\scriptsize{}6}\tabularnewline[\doublerulesep]
\hline 
\noalign{\vskip\doublerulesep}
{\small{}$X_{1},\ X_{2}$} & {\scriptsize{}$\begin{array}{c}
x\left(x\partial_{x}+y\partial_{y}\right)+\cdots\\
y\left(x\partial_{x}+y\partial_{y}\right)+\cdots
\end{array}$} & {\scriptsize{}$\begin{array}{c}
\left(x+\frac{29}{15}y\right)\left(x\partial_{x}+y\partial_{y}\right)+\cdots\\
x^{2}\left(x\partial_{x}+y\partial_{y}\right)+\cdots
\end{array}$}\tabularnewline[\doublerulesep]
\hline 
\noalign{\vskip\doublerulesep}
{\small{}$\nu\left(X_{1}\right),\ \nu\left(X_{2}\right)$} & {\scriptsize{}$2,\ 2$} & {\scriptsize{}$2,\ 3$}\tabularnewline[\doublerulesep]
\hline 
\noalign{\vskip\doublerulesep}
{\small{}Type} & {\scriptsize{}$\left(\mathfrak{O}_{d}\right)$ - $1$ free point} & {\scriptsize{}$\left(\mathfrak{E}_{d}\right)$ - $1$ free point}\tabularnewline[\doublerulesep]
\end{tabular}\vspace{0.6cm}

\caption{\label{tab:Exemples-of-different}Examples of different types of Saito
bases.}
\end{table}

\subsection{Base of type $\left(\mathfrak{E}_{d}^{\prime}\right)$ and $\left(\mathfrak{O}_{d}^{\prime}\right)$.}

Beyond Example \ref{Example.5.droites}, the curve $S$ defined by
\[
S=\left\{ y^{5}+x^{5}+x^{6}=0\right\} 
\]
belongs to the generic component of the moduli space of five smooth
and transversal curves. An optimal vector field $X_{1}$ for $S$
can be written 
\[
X_{1}=\left(\frac{1}{5}xy-\frac{1}{25}x^{2}y+\frac{6}{125}x^{3}y+\frac{36}{125}x^{4}y\right)\partial_{x}+\left(\frac{1}{5}y^{2}+\frac{216}{625}x^{3}y^{2}\right)\partial_{y}
\]
whose initial part is
\begin{equation}
\frac{y}{5}\left(x\partial_{x}+y\partial_{y}\right).\label{initial.part}
\end{equation}
Thus $X_{1}$ is dicritical and of multiplicity $2$. However, after
one blowing-up, $X_{1}^{E}$ is not transverse to $D$ at every point
: indeed, following (\ref{initial.part}), it is tangent to $D$ at
the point corresponding to the direction $y=0.$ To formalize these
remarks, let us recall the following definition
\begin{defn}
Let $D$ be a divisor and $X$ a vector field defined in a neighborhood
of $D$ that does not leave invariant $D.$ The locus of tangency
between $X$ and $D$ is the common zeros of $F$ and $X\cdot F$
where $F$ is any local equation of $D.$ It is denoted by 
\[
\textup{Tan\ensuremath{\left(X,D\right)}.}
\]
\end{defn}
By definition, the locus of tangency between $D$ and $X$ contains
the singular points of $X$ which are on $D.$ In the example (\ref{initial.part}),
we have 
\[
\textup{Tan}\left(X_{1}^{E},D\right)=\left\{ \left(x_{1}=0,y_{1}=0\right)\right\} \neq\textup{Tan}\left(S^{E},D\right)=\emptyset.
\]
This leads us to introduce the following notion.
\begin{defn}
A curve $S$ of radial type is said to be of \emph{pure radial} \emph{type}
if for any optimal vector field $X_{1}$ the following equality holds
\[
\textup{Tan}\left(X_{1}^{E},D\right)=\textup{Tan}\left(S^{E},D\right).
\]
If $S$ is not pure radial, then the non empty set
\[
\textup{Tan}\left(X_{1}^{E},D\right)\setminus\textup{Tan}\left(S^{E},D\right)
\]
is called the set of \emph{free points of $X_{1}$}
\end{defn}
Notice that by construction of $X_{1}$, in any case, the inclusion
\[
\textup{Tan}\left(S^{E},D\right)\subset\textup{Tan}\left(X_{1}^{E},D\right)
\]
holds. The main feature of this definition relies on the fact that
it allows to state a characterization of the curves admitting a basis
of type $\left(\mathfrak{E}_{d}^{\prime}\right)$ or $\left(\mathfrak{O}_{d}^{\prime}\right)$.
\begin{thm}
\label{thm:pure.radial}The following properties are equivalent :
\begin{enumerate}
\item $S$ is of pure radial type.
\item $S$ admits a Saito basis of type $\left(\mathfrak{E}_{d}^{\prime}\right)$
or $\left(\mathfrak{O}_{d}^{\prime}\right)$.
\end{enumerate}
\end{thm}
\begin{proof}
We begin by proving $\left(2\right)\Longrightarrow\left(1\right)$.
Assume that $S$ admits an adapted Saito basis of type $\left(\mathfrak{E}_{d}^{\prime}\right)$
or $\left(\mathfrak{O}_{d}^{\prime}\right)$. According to Remark
\ref{relevement.base.adaptee}, for any point $c\in D,$ the family
\[
\left\{ \left(X_{1}^{E}\right)_{c},\left(X_{2}^{E}\right)_{c}\right\} 
\]
is a Saito basis of the germ of curve $\left(S^{E}\right)_{c}$. Let
$c\in D\setminus\textup{Tan}\left(S^{E},D\right).$ Suppose first
that $c\notin\textup{Sing}\left(E^{-1}\left(S\right)\right)$. Then
following Remark \ref{relevement.base.adaptee}, the product $X_{1}^{E}\wedge X_{2}^{E}$
is a unity at $c.$ Now, $X_{1}$ is dicritical and $X_{2}$ is not.
Thus in local coordinates $\left(x,y\right)$ at $c$ in which $x=0$
is local equation of $D,$ we can write 
\begin{align*}
X_{1}^{E}\wedge X_{2}^{E} & =\left(u\partial x+v\partial y\right)\wedge\left(ax\partial x+b\partial y\right),\qquad u,v,a,b\in\mathbb{C}\left\{ x,y\right\} \\
 & =avx-bu.
\end{align*}
Therefore $u$ is a unity and $X_{1}^{E}$ is transverse to $D.$
Suppose now that $c\in\textup{Sing}\left(E^{-1}\left(S\right)\right).$
Since $c\in D\setminus\textup{Tan}\left(S^{E},D\right)$ then $S^{E}$
is regular and transverse to $D$. Now, considering local coordinates
$\left(x,y\right)$ in which $xy=0$ is a local equation of $E^{-1}\left(S\right)$,
we obtain 
\begin{align*}
X_{1}^{E}\wedge X_{2}^{E} & =\left(u\partial x+vy\partial y\right)\wedge\left(ax\partial x+by\partial y\right)\qquad u,v,a,b\in\mathbb{C}\left\{ x,y\right\} \\
 & =avxy-buy
\end{align*}
which has to be of the form $\left(\textup{unity}\right)\times y$
according to the criterion of Saito. Therefore, $u$ is a unity and
$X_{1}^{E}$ is still transverse to $D$, which completes the proof
of the equality 
\[
\textup{Tan}\left(X_{1}^{E},D\right)=\textup{Tan}\left(S^{E},D\right).
\]
We now proceed to the proof of $\left(1\right)\Longrightarrow\left(2\right)$.
Let $\left\{ X_{1},X_{2}\right\} $ be an adapted Saito basis for
$S.$ The curve $S$ being radial, let us write
\begin{equation}
X_{1}=h_{1}\left(x\partial x+y\partial y\right)+\cdots.\label{ecriture.de.X1}
\end{equation}
The hypothesis is equivalent to assume that the tangent cone of $h_{1}$
coincide with the locus of tangency $\textup{Tan}\left(S^{E},D\right)$
for any optimal vector field $X_{1}$. Assume first that $\nu\left(S\right)$
is odd. According to Proposition (\ref{SaitoBasisForm}), the valuation
of $X_{1}$ is 
\[
\nu\left(X_{1}\right)=\frac{\nu-1}{2}.
\]
If $X_{2}$ is not dicritical, then $\nu\left(X_{2}\right)=\frac{\nu+1}{2}$.
Therefore, the basis $\left\{ X_{1},X_{2}\right\} $ is of type $\left(\mathfrak{O}_{d}^{\prime}\right)$
and the proposition is proved. Assume $X_{2}$ is dicritical and $\nu\left(X_{2}\right)=\frac{\nu-1}{2}$.
As in (\ref{ecriture.de.X1}), we write 
\[
X_{2}=h_{2}\left(x\partial x+y\partial y\right)+\cdots
\]
and
\begin{align*}
h_{2} & =q_{2}\cdot\overline{h_{2}}
\end{align*}
where the tangent cone of $q_{2}$ does not meet $\textup{Tan\ensuremath{\left(S^{E},D\right)}. }$
For any value of $\alpha$ and $\beta$, the initial part of $\alpha X_{1}+\beta X_{2}$
is written 
\[
\left(\alpha h_{1}+\beta q_{2}\overline{h_{2}}\right)\left(x\partial_{x}+y\partial_{y}\right).
\]
The hypothesis ensures that the tangent cone of $\alpha h_{1}+\beta q_{2}\overline{h_{2}}$
is in $\textup{Tan}\left(S^{E},D\right)$. Since the tangent cone
of $h_{1}$ is in $\textup{Tan}\left(S^{E},D\right)$, it can be seen
that the function $q_{2}$ is constant and that there exists a constant
$u$ such that 
\[
h_{2}=uh_{1}
\]
Then the basis $\left\{ X_{1},X_{2}-uX_{1}\right\} $ is of type $\left(\mathfrak{O}_{d}^{\prime}\right).$

Assume finally that $\nu\left(S\right)$ is even and consider a smooth
curve $l$ which is attached to a point in $\textup{Tan}\left(S^{E},D\right)$
after on blowing-up. Let $\left\{ X_{1}^{l},X_{2}^{l}\right\} $ be
an adapted Saito basis for $S\cup l.$ Consider some coordinates in
which $l=\left\{ x=0\right\} $ and write 
\[
X_{i}^{l}=xa_{i}\partial_{x}+\left(y^{\alpha_{i}}b_{i}^{0}\left(y\right)+xb_{i}^{1}\right)\partial_{y}
\]
with $b_{i}\left(0\right)\neq0.$ Since $\nu\left(S\cup l\right)$
is odd, a few cases may occur :
\begin{enumerate}
\item Assume the basis is of type $\left(\mathfrak{O}\right)$. Then, we
can suppose that $\nu\left(X_{1}^{l}\right)=\nu\left(X_{2}^{l}\right)=\frac{\nu\left(S\right)}{2}$
and $\alpha_{1}=\alpha_{2}$. The family 
\[
\left\{ X_{1}^{l},\overline{X}_{2}^{l}=\frac{1}{x}\left(X_{2}^{l}-\frac{b_{2}^{0}}{b_{1}^{0}}X_{1}^{l}\right)\right\} 
\]
is a Saito basis for $S$ with 
\[
\nu\left(X_{1}^{l}\right)=\frac{\nu}{2}\qquad\textup{and}\qquad\nu\left(\overline{X}_{2}^{l}\right)\geq\frac{\nu}{2}-1.
\]
If $\nu\left(\overline{X}_{2}^{l}\right)\geq\frac{\nu}{2}$ then $\mathfrak{s}\left(S\right)\geq\frac{\nu\left(S\right)}{2}$
which is impossible. Thus $\nu\left(\overline{X}_{2}^{l}\right)=\frac{\nu}{2}-1$.
But then, following Lemma \ref{lem:0or-1} $\overline{X}_{2}^{l}$
cannot be dicritical which contradicts the radiality of $S.$ Finally,
the Saito basis of $S\cup l$ cannot be of type $\left(\mathfrak{O}\right).$
\item Assume it is of type $\left(\mathfrak{O}_{d}\right)$ but not of type
$\left(\mathfrak{O}_{d}^{\prime}\right)$. Applying Theorem \ref{thm:pure.radial}
to this case for which $\nu\left(S\cup l\right)$ is odd ensures that
$S\cup l$ is not pure radial. Thus up to some changes of basis, $\textup{Tan}\left(\left(X_{1}^{l}\right)^{E},S\cup l\right)$
and $\textup{Tan}\left(\left(X_{2}^{l}\right)^{E},S\cup l\right)$
contains some points out of $\textup{Tan}\left(\left(S\cup l\right)^{E},D\right)=\textup{Tan}\left(S^{E},D\right).$
Therefore, we obtain a Saito basis for $S$ of the form 
\[
\left\{ X_{1}^{l},\overline{X}_{2}^{l}\right\} 
\]
where the tangent cone of $\overline{X}_{2}^{l}$ is not contained
in $\textup{Tan}\left(S^{E},D\right)$, which contradicts the assumption
of $S$ being pure radial. Therefore, the Saito basis of $S\cup l$
cannot be of type $\left(\mathfrak{O}_{d}\right)$ but not of type
$\left(\mathfrak{O}_{d}^{\prime}\right)$.
\item Finally, $S\cup l$ admits a Saito basis $\left\{ X_{1}^{l},X_{2}^{l}\right\} $
of type $\left(\mathfrak{O}_{d}^{\prime}\right)$ with $\nu\left(X_{1}^{l}\right)=\nu\left(X_{2}^{l}\right)-1=\frac{\nu}{2}.$
If $\alpha_{2}>\alpha_{1}$ then 
\[
\left\{ X_{1}^{l},\overline{X}_{2}^{l}=\frac{1}{x}\left(X_{2}^{l}-y^{\alpha_{2}-\alpha_{1}}\frac{b_{2}^{0}}{b_{1}^{0}}X_{1}^{l}\right)\right\} 
\]
is a Saito basis for $S$ with 
\[
\nu\left(X_{1}^{l}\right)=\frac{\nu}{2}\qquad\textup{and}\qquad\nu\left(\overline{X}_{2}^{l}\right)\geq\frac{\nu}{2}
\]
which is impossible. Thus $\alpha_{2}\leq\alpha_{1}$ and 
\[
\left\{ \overline{X}_{1}^{l}=\frac{1}{x}\left(X_{1}^{l}-y^{\alpha_{1}-\alpha_{2}}\frac{b_{1}^{0}}{b_{2}^{0}}X_{2}^{l}\right),X_{2}^{l}\right\} 
\]
is a Saito basis for $S$ of type $\left(\mathfrak{E}_{d}^{\prime}\right)$.
\end{enumerate}
\end{proof}
Finally, from the proof above we deduce the following
\begin{lem}
\label{Lemme.O.E}If $S$ is of type $\left(\mathfrak{O}\right)$
then $S\cup l$ is of type $\left(\mathfrak{E}\right).$ If $S$ is
of type $\left(\mathfrak{E}_{d}\right)$ or $\left(\mathfrak{E}_{d}^{\prime}\right)$
then $S\cup l$ is of type $\left(\mathfrak{O}_{d}\right)$ or $\left(\mathfrak{O}_{d}^{\prime}\right)$.
\end{lem}

\subsection{Cohomology of $\Theta_{S}$}

As we will explain in the next section, the cohomology of the sheaf
$\Theta_{S}$ computes the generic dimension of $\moduli$. The associated
formula depends on the type of Saito basis of $S.$
\begin{prop}
\label{dimension.coho}The dimension of the cohomology group $H^{1}\left(D,\Theta_{S}\right)$
can be obtained from the multiplicities of an adapted Saito basis
of $S$ the following way
\begin{enumerate}
\item If $\nu\left(X_{1}\right)+\nu\left(X_{2}\right)=\nu\left(S\right)$
then 
\[
\dim H^{1}\left(D,\Theta_{S}\right)=\frac{\left(\nu_{1}-1\right)\left(\nu_{1}-2\right)}{2}+\frac{\left(\nu_{2}-1\right)\left(\nu_{2}-2\right)}{2}
\]
\item If $\nu\left(X_{1}\right)+\nu\left(X_{2}\right)=\nu\left(S\right)-1$
then 
\[
\dim H^{1}\left(D,\Theta_{S}\right)=\frac{\left(\nu_{1}-1\right)\left(\nu_{1}-2\right)}{2}+\frac{\left(\nu_{2}-1\right)\left(\nu_{2}-2\right)}{2}+\nu\left(S\right)-2-\nu_{0}
\]
where $\text{\ensuremath{\nu_{i}}}=\nu\left(X_{i}\right),\ i=1,2$
and $\nu_{0}=\nu\left(\gcd\left(X_{1}^{\left(\nu\left(X_{1}\right)\right)},X_{2}^{\left(\nu\left(X_{2}\right)\right)}\right)\right).$
\end{enumerate}
\end{prop}
\begin{proof}
The proof of the first equality is in \cite{YoyoBMS}. Below, we only
give a proof of the second equality. Let us consider the standard
system of coordinates defined in a neighborhood of $D$ and introduced
in section \ref{subsection.Generic.value}.

One can compute the cohomology using the associated covering and thus
\begin{equation}
H^{1}\left(D,\Theta_{S}\right)=H^{1}\left(\left\{ U_{1},U_{2}\right\} ,\Theta_{S}\right)=\frac{H^{0}\left(U_{1}\cap U_{2},\Theta_{S}\right)}{H^{0}\left(U_{1},\Theta_{S}\right)\oplus H^{0}\left(U_{2},\Theta_{S}\right)}.\label{Equation.homologique}
\end{equation}

The task is now to describe in coordinates each $H^{0}$ involved
in the quotient above. To deal with $H^{0}\left(U_{1},\Theta_{S}\right)$,
we start with the criterion of Saito 
\begin{equation}
X_{1}\wedge X_{2}=uf.\label{Saito.relation}
\end{equation}
As $\nu_{1}+\nu_{2}=\nu\left(S\right)-1$, blowing-up the criterion
of Saito in the first chart $\left(x_{1},y_{1}\right)$ yields
\[
\underbrace{\frac{E^{\star}X_{1}}{x_{1}^{\nu_{1}-1}}}_{X_{1}^{1}}\wedge\underbrace{\frac{E^{\star}X_{2}}{x_{1}^{\nu_{2}-1}}}_{X_{2}^{1}}=u\circ E\cdot x_{1}^{2}\frac{f\circ E}{x_{1}^{\nu\left(S\right)}},
\]
Let $Y$ be a section of $\Theta_{S}$ on $U_{1}.$ By definition,
there exists $g_{1}\in\mathcal{O}\left(U_{1}\right)$ such that 
\[
Y\wedge X_{1}^{1}=g_{1}x_{1}\frac{f\circ E}{x_{1}^{\nu\left(S\right)}}.
\]
Hence, one has 
\[
\left(x_{1}Y-g_{1}\frac{1}{u\circ E}X_{2}^{1}\right)\wedge X_{1}^{1}=0.
\]
Assume $X_{1}$ is not dicritical. Then, $X_{1}^{1}$ has only isolated
singularities and there exists $h_{1}\in\mathcal{O}\left(U_{1}\right)$
such that 
\[
x_{1}Y=g_{1}\frac{1}{u\circ E}X_{2}^{1}+h_{1}X_{1}^{1}.
\]
If now $X_{1}$ is dicritical, then $\frac{X_{1}^{1}}{x_{1}}$ extends
analytically along $D$ and has only isolated singularities. Therefore,
there still exists $h_{1}\in\mathcal{O}\left(U_{1}\right)$ such that
\[
x_{1}Y=k_{1}\frac{1}{u\circ E}X_{2}^{1}+\frac{h_{1}}{x_{1}}X_{1}^{1}.
\]
Since, $x_{1}Y$ and $X_{2}^{1}$ are tangent to $D,$ $x_{1}$ divides
$h_{1}.$ Thus we get 
\[
H^{0}\left(U_{1},\Theta_{S}\right)=\left\{ \left.Y=\frac{1}{x_{1}}\left(\phi_{1}^{1}X_{1}^{1}+\phi_{2}^{1}X_{2}^{1}\right)\right|\begin{array}{l}
\left(1\right)~\phi_{i}^{1}\in\mathcal{O}\left(U_{1}\right)\\
\left(2\right)\ \textup{\ensuremath{Y} extends analytically along \ensuremath{D}}
\end{array}\right\} .
\]
We now proceed to analyse the second condition highlighted above :
let us write the expansion of $X_{i}$ in homogeneous components 
\[
X_{i}=X_{i}^{\left(\nu_{i}\right)}+X_{i}^{\left(\nu_{i}+1\right)}+\cdots
\]
The relation $\nu_{1}+\nu_{2}=\nu-1$ implies that 
\[
X_{1}^{\left(\nu_{1}\right)}\wedge X_{2}^{\left(\nu_{2}\right)}=0
\]
and we can write 
\[
X_{i}^{\left(\nu_{i}\right)}=\delta_{i}X_{0}
\]
where $X_{0}=\gcd\left(X_{1}^{\left(\nu_{1}\right)},X_{2}^{\left(\nu_{2}\right)}\right)$
and $\delta_{i},i=1,2$ are homogeneous functions such that 
\[
\delta_{1}\wedge\delta_{2}=1.
\]
The expression of $Y$ can be expanded with respect to $x_{1}$ in
\begin{align*}
Y & =\frac{1}{x_{1}}\sum_{i=1,2}\underbrace{\left(\phi_{i}^{1,0}\left(y_{1}\right)+x_{1}\left(\cdots\right)\right)}_{\phi_{i}^{1}}\underbrace{\left(\delta_{i}^{1}X_{0}^{1}+x_{1}\left(\ldots\right)\right)}_{X_{i}^{1}},\qquad\delta_{i}^{1}=\frac{\delta_{i}\circ E}{x_{1}^{\nu\left(\delta_{i}\right)}}
\end{align*}
Thus the condition $Y$ being extendable along $D$ reduces to 
\[
\sum_{i=1,2}\phi_{i}^{1,0}\delta_{i}^{1}=0.
\]
We proceed analogously for the open sets $U_{2}$ and $U_{1}\cap U_{2}$
and obtain the following description where the exponent $2$ refers
to the second chart $\left(x_{2},y_{2}\right)$
\begin{align}
H^{0}\left(U_{1},\Theta_{S}\right) & =\left\{ \left.Y^{1}=\frac{1}{x_{1}}\sum_{i=1,2}\phi_{i}^{1}X_{i}^{1}\right|\begin{array}{l}
\phi_{i}^{1}\in\mathcal{O}\left(U_{1}\right)\\
\sum_{i=1,2}\phi_{i}^{1,0}\delta_{i}^{1}=0
\end{array}\right\} ,\nonumber \\
H^{0}\left(U_{2},\Theta_{S}\right) & =\left\{ \left.Y^{2}=\frac{1}{y_{2}}\sum_{i=1,2}\phi_{i}^{2}X_{i}^{2}\right|\begin{array}{l}
\phi_{i}^{2}\in\mathcal{O}\left(U_{2}\right)\\
\sum_{i=1,2}\phi_{i}^{2,0}\delta_{i}^{2}=0
\end{array}\right\} ,\label{cohomologie.calcul}\\
H^{0}\left(U_{1}\cap U_{2},\Theta_{S}\right) & =\left\{ \left.Y^{12}=\frac{1}{x_{1}}\sum_{i=1,2}\phi_{i}^{12}X_{i}^{1}\right|\begin{array}{l}
\phi_{i}^{12}\in\mathcal{O}\left(U_{1}\cap U_{2}\right)\\
\sum_{i=1,2}\phi_{i}^{12,0}\delta_{i}^{1}=0
\end{array}\right\} .\nonumber 
\end{align}
We may now compute the number of obstructions involved in the cohomological
equation describing the quotient (\ref{Equation.homologique}), namely,
\[
Y^{12}=Y^{2}-Y^{1}.
\]
In view of the description above, the cohomological equation splits
into the system
\[
\phi_{i}^{12}=\frac{\phi_{i}^{2}}{y_{1}^{\nu_{i}}}-\phi_{i}^{1},\ i=1,2
\]
which we filter with respect to $x_{1}$ obtaining
\begin{align}
\phi_{i}^{12,0} & =\frac{\phi_{i}^{2,0}}{y_{1}^{\nu_{i}}}-\phi_{i}^{1,0},\ i=1,2\label{cohomologie.0}\\
\phi_{i}^{12,1} & =\frac{\phi_{i}^{2,1}}{y_{1}^{\nu_{i}}}-\phi_{i}^{1,1},\ i=1,2\label{cohomologie.1}
\end{align}
where 
\[
\phi_{i}^{\star}=\phi_{i}^{\star,0}+x_{1}\phi_{i}^{\star,1},
\]
with $\star=1,\ 2,\ 12$. Let us analyse the system (\ref{cohomologie.0}).
Since the functions $\delta_{i}$ are relatively prime, the conditions
involved in the description of the cohomological spaces (\ref{cohomologie.calcul})
ensures that there exist analytical functions $\dot{\phi}^{\star,0}$
such that 
\[
\phi_{1}^{\star,0}=\dot{\phi}^{\star,0}\delta_{2}^{\star}\textup{ and }\phi_{2}^{\star,0}=-\dot{\phi}^{\star,0}\delta_{1}^{\star}.
\]
for $\star=1,\ 2,\ 12$. Thus, the system (\ref{cohomologie.0}) reduces
to the sole equation
\[
\dot{\phi}^{12,0}=\frac{\dot{\phi}^{2,0}}{y_{1}^{\nu_{1}+\nu\left(\delta_{2}\right)}}-\dot{\phi}^{1,0}.
\]

Writing the Laurent expansions of the above functions yields the relation
\[
\sum_{k\in\mathbb{Z}}\dot{\phi}_{k}^{12}y_{1}^{k}=\sum_{k\in\mathbb{N}}\dot{\phi}_{k}^{2}y_{1}^{-k-\nu_{1}-\nu\left(\delta_{2}\right)}-\sum_{k\in\mathbb{N}}\dot{\phi}_{k}^{1}y_{1}^{k}
\]
which implies $\dot{\phi}_{k}^{12}=0$ for $-\nu_{1}-\nu\left(\delta_{2}\right)+1\leq k\leq-1.$
These conditions provide the sole $-\nu_{1}-\nu\left(\delta_{2}\right)+1$
obstructions to the cohomological equation (\ref{Equation.homologique}).
The system (\ref{cohomologie.1}) involves two independant cohomological
equations. We can proceed analogously to identify $\frac{\left(\nu_{i}-1\right)\left(\nu_{i}-2\right)}{2}$
obstructions $i=1,2$ , respectively, for the equation $i=1,2.$ Finally,
the formula $\left(2\right)$ of the proposition follows from the
relation 
\[
\nu\left(\delta_{2}\right)=\nu_{2}-\nu_{0}.
\]
\end{proof}
Notice that the second case of Proposition \ref{division.saito.process}
may occur when $S$ is of type $\left(\mathfrak{O}_{d}\right)$ or
$\left(\mathfrak{E}_{d}\right)$. In that case, it can be seen that
$\nu\left(X_{1}\right)-\nu_{0}$ is the number of free points of $X_{1}.$

\section{Dimension of the moduli space of a singular regular point.}

In this section, we intend to apply the previous results to compute
the generic dimension of $\moduli$ for
\[
S=\left\{ x^{n}+y^{n}=0\right\} ,
\]
and thus, to recover a classical result due to Granger \cite{Granger}.

To achieve this, we have to identify precisely the topology of the
generic optimal vector field for $S.$ First, we are going to improve
somehow the optimality of the generic optimal vector field studying
when this optimality is preserved after one blowing-up.

\subsection{Optimality after one blowin-up.}
\begin{prop}
\label{optimal.avant.apr=0000E8s}Let $S$ be a generic curve in its
moduli space. Let $c\in D$ a point in the exceptional divisor of
the single blowing-up $E$. Assume that
\begin{quote}
$\left(\star\right)\ $ there exists a germ of regular curve $l$
such that $\left(S^{E}\right)_{c}\cup\left(l^{E}\right)_{c}$ has
no Saito basis of type $\left(\mathfrak{E}_{d}^{\prime}\right)$.
\end{quote}
Then there exists a vector field $X$ optimal for $S$ such that $\left(X^{E}\right)_{c}$
is optimal for $\left(S^{E}\right)_{c}$.
\end{prop}
\begin{proof}
Let $\left\{ X_{1},X_{2}\right\} $ be an adapted Saito basis for
$S.$ If $\nu\left(X_{1}\right)=\nu\left(X_{2}\right)$ which is satisfy
when the basis is of type $\left(\mathfrak{E}\right)$, $\left(\mathfrak{O}_{d}\right)$
then for $\alpha$ and $\beta$ generic one has
\[
\alpha X_{1}^{E}+\beta X_{2}^{E}=\left(\alpha X_{1}+\beta X_{2}\right)^{E}.
\]
According to Remark (\ref{relevement.base.adaptee}), $\left\{ \left(X_{1}^{E}\right)_{c},\left(X_{2}^{E}\right)_{c}\right\} $
is a Saito basis for $\left(S^{E}\right)_{c},$ therefore at $c$
one has 
\[
\nu_{c}\left(\left(\alpha X_{1}+\beta X_{2}\right)^{E}\right)=\mathfrak{s}\left(\left(S^{E}\right)_{c}\right)
\]
Thus, in that case, choosing $X=\alpha Y_{1}+\beta Y_{2}$ yields
the lemma.

Now, assume that $\nu\left(X_{1}\right)<\nu\left(X_{2}\right).$ Suppose
first that $\nu\left(S\right)$ is odd then $S$ is of type $\left(\mathfrak{O}\right).$
Let us consider a curve $l$ satisfying the hypothesis of the lemma.
According to Lemma \ref{Lemme.O.E}, an adapated Saito basis $\left\{ X_{1}^{l},X_{2}^{l}\right\} $
is of type $\left(\mathfrak{E}\right)$ with 
\[
\text{\ensuremath{\nu\left(X_{1}^{l}\right)=\nu\left(X_{2}^{l}\right)=\frac{\nu\left(S\right)+1}{2}}}.
\]
Applying the process of division, we are lead to an adapted Saito
basis $\left\{ \overline{X}_{1}^{l},X_{2}^{l}\right\} $ for $S$
with 
\begin{equation}
\nu\left(\overline{X}_{1}^{l}\right)=\frac{\nu\left(S\right)-1}{2}<\nu\left(X_{2}^{l}\right)=\frac{\nu\left(S\right)+1}{2}\label{equation.premiere.inegalite}
\end{equation}
The blow-up family $\left\{ \left(\overline{X}_{1}^{l}\right)_{c}^{E},\left(X_{2}^{l}\right)_{c}^{E}\right\} $
is a basis for $\left(S^{E}\right)_{c}.$ Now, suppose that 
\[
\nu_{c}\left(\left(\overline{X}_{1}^{l}\right)^{E}\right)\geq\nu_{c}\left(\left(X_{2}^{l}\right)^{E}\right)+1
\]
therefore,
\[
\nu_{c}\left(L\left(\overline{X}_{1}^{l}\right)^{E}\right)\geq\nu_{c}\left(\left(X_{2}^{l}\right)^{E}\right)+2
\]
where $L$ is a local equation of $l^{E}.$ The family $\left\{ L\left(\overline{X}_{1}^{l}\right)^{E},\left(X_{2}^{l}\right)^{E}\right\} $
is a Saito basis for $S^{E}\cup l^{E}$ at $c$ and following Remark
\ref{rem:EDprime} it is of type $\left(\mathfrak{E}_{d}^{\prime}\right)$.
That is impossible. Hence, 
\begin{equation}
\nu_{c}\left(\left(\overline{X}_{1}^{l}\right)^{E}\right)\leq\nu_{c}\left(\left(X_{2}^{l}\right)^{E}\right)\label{equation.seconde.inegalite}
\end{equation}
and, according to (\ref{equation.premiere.inegalite}) and (\ref{equation.seconde.inegalite}),
$X=\overline{X}_{1}^{l}$ satisfies the conclusion of the lemma. Finally,
if $\nu\left(S\right)$ is even then $S$ is of type $\left(\mathfrak{E}_{d}\right)$.
Therefore, $S\cup l$ is of type $\left(\mathfrak{O}_{d}\right)$
and the arguments are similar.
\end{proof}
\begin{cor}
\label{optimal.avant.apres.coro} If any component of $S^{E}$ satisfies
the hypothesis $\left(\star\right)$ of Proposition \ref{optimal.avant.apr=0000E8s},
then there exists a vector field $X$ optimal for $S$ such that,
for any $c,$ $\left(X^{E}\right)_{c}$ is optimal for $\left(S^{E}\right)_{c}$.
\end{cor}
\begin{proof}
Indeed, for any point $c$ in the tangent cone of $S,$ consider $X_{c}$
given by Proposition \ref{optimal.avant.apr=0000E8s} for the curve
$\left(S^{E}\right)_{c}$. Then for a generic family of complex numbers
$\left\{ \alpha_{c}\right\} $, the vector field
\[
X=\sum\alpha_{c}X_{c}
\]
satisfies the property.
\end{proof}

\subsection{Dimension of $\protect\moduli$ where $S=\left\{ x^{n}+y^{n}=0\right\} $}

The curve $S$ is desingularized by a single blowing-up. From \cite{MatQuasi},
the generic dimension of $\moduli$ is equal to 
\[
\dim H^{1}\left(D,\Theta_{S}\right).
\]
Following Proposition \ref{dimension.coho}, it can be computed from
some topological data associated to an adapted basis of Saito for
$S.$ Below, we are going to describe these bases according to the
value of $n.$

If $n=3$ then there are coordinates $\left(x,y\right)$ in which
$S=\left\{ f=xy\left(x+y\right)=0\right\} .$ The family 
\[
\left\{ X_{1}=x\partial_{x}+y\partial_{y},\ X_{2}=\sharp df=\partial_{x}f\partial_{y}-\partial_{y}f\partial_{x}\right\} 
\]
is a Saito basis for $S.$ Since $\nu\left(X_{1}\right)=\nu\left(X_{2}\right)-1=1$,
$X_{1}$ is dicritical but not $X_{2}$, $S$ is of type $\left(\mathfrak{O}_{d}^{\prime}\right).$
If $n=4$ then there are coordinates $\left(x,y\right)$ in which
$S=\left\{ f=xy\left(x+y\right)\left(x+ay\right)=0\right\} $ for
some $a\notin\left\{ 0,1\right\} .$ Hence, the family 
\[
\left\{ X_{1}=x\partial_{x}+y\partial_{y},\ X_{2}=\sharp df\right\} 
\]
is a Saito basis for $S.$ Since $\nu\left(X_{1}\right)=\nu\left(X_{2}\right)-2=1$,
$X_{1}$ is dicritical but not $X_{2}$, $S$ is of type $\left(\mathfrak{E}_{d}^{\prime}\right).$
In the latter case, the dimension of $\moduli$ is $1.$

Now suppose $n\geq5.$
\begin{prop}
\label{thm.one.blowing.up}The curve $S$ is of type $\left(\mathfrak{O}_{d}\right)$
or $\left(\mathfrak{E}_{d}\right)$. Moreover, the generic optimal
vector $X_{1}$ is completely regular after a single blowing-up and
has $\left\lceil \frac{n}{2}\right\rceil -2$ free points.
\end{prop}
\begin{proof}
Following notations introduced in \cite[p. 657]{hertlingfor} for
a germ at $p$ of vector field $X$ and a germ of curve $S$ given
in coordinates by 
\[
X=a\left(x,y\right)\partial_{x}+b\left(x,y\right)\partial_{y}\qquad S=\left\{ x=0\right\} 
\]
 we recall the following definitions :
\begin{enumerate}
\item if $S$ is invariant by $X,$ the integer $\nu_{y}\left(b\left(0,y\right)\right)$
is called \emph{the index} of $X$ at $p$ with respect to $S$ and
it is denoted by 
\[
\textup{ind}\left(X,S,p\right).
\]
\item if $S$ is not invariant by $X,$ the integer $\nu_{y}\left(a\left(0,y\right)\right)$
is called \emph{the tangency order} of $X$ at $p$ with respect to
$S$ and it is denoted by
\[
\textup{tan}\left(X,S,p\right).
\]
\end{enumerate}
Suppose $X_{1}$ non dicritical, then according to \cite[Lemma 1]{hertlingfor},
one has
\begin{equation}
\nu\left(X_{1}\right)+1=\sum_{c\in D}\textup{ind}\left(X_{1}^{E},D,c\right).\label{equation.index}
\end{equation}
For any point $c$ in the tangent cone of $S$, the curve $\left(S^{E}\cup D\right)_{c}$
is a union of two transversal smooth curves. Therefore, the index
$\textup{ind}\left(X_{1}^{E},D,c\right)$ is at least $1$ since $\left(X_{1}^{E}\right)_{c}$
is singular. Therefore, one has 
\begin{equation}
\sum_{c\in D}\textup{ind}\left(X_{1}^{E},D,c\right)\geq\sharp\textup{tangent cone}=n.\label{inequality.index}
\end{equation}
On the other hand, the optimality of $X_{1}$ ensures that 
\begin{equation}
\nu\left(X_{1}\right)\leq\frac{n}{2}.\label{inequality.multiplicity}
\end{equation}
The equality \ref{equation.index} and the inequalities (\ref{inequality.index})
and (\ref{inequality.multiplicity}) are incompatible with $n\geq5$,
and thus $X_{1}$ is dicritical. Any component of $S^{E}$ is a regular
curve. Since the union of two curves is not of type $\left(\mathfrak{E}_{d}^{\prime}\right)$,
any component of $S^{E}$ satisfies the hypothesis $\left(\star\right)$
of Propostion \ref{optimal.avant.apr=0000E8s}. As a consequence,
we can consider $X_{1}$ to be not only optimal for $S$ but also
optimal after one blowing-up. Since any component of $S^{E}$ are
regular curve, whose Saito number are equal to $0,$ the vector field
$X_{1}^{E}$ is regular at the tangent cone of $S.$ Moreover, there
exists $Y$ such that $\left\{ X_{1},Y\right\} $ is an adapted Saito
basis. Thus, after one blowing-up, one can write 
\[
X_{1}^{E}\wedge Y^{E}=u\circ E\frac{f\circ E}{x_{1}^{n}}
\]
where $f=x^{n}+y^{n}$. Since out of the tangent cone of $S,$ the
function $\frac{f\circ E}{x_{1}^{n}}$ is a unit, $X_{1}^{E}$ is
finally regular at any point of $D$. 

Following again \cite[Lemma 1]{hertlingfor}, one has 
\[
\nu\left(X_{1}\right)+1=\left\lceil \frac{n}{2}\right\rceil =2+\sum_{c\in D}\textup{tan}\left(X_{1}^{E},D,c\right).
\]
The above relation concludes the proof of the proposition : $\textup{Tan}\left(S^{E},D\right)$
being empty, any tangency point between $X_{1}^{E}$ and $D$ is a
free point. 
\end{proof}
As a consequence of Proposition \ref{thm.one.blowing.up}, we recover
a classical result of Granger concerning the generic dimension of
the moduli space of $S$ \cite{Granger}. According to Theorem \ref{thm.one.blowing.up},
the Saito basis of $S$ satisfies 
\[
\nu\left(X_{1}\right)=\left\{ \begin{array}{ll}
\frac{n}{2}-1 & \textup{if \ensuremath{n} is even}\\
\frac{n-1}{2} & \textup{else}
\end{array}\right.\textup{ and }\nu\left(X_{2}\right)=\left\{ \begin{array}{ll}
\frac{n}{2} & \textup{if \ensuremath{n} is even}\\
\frac{n-1}{2} & \textup{else}
\end{array}\right..
\]
Moreover, by construction, the integer $\nu_{0}$ identified in Proposition
\ref{dimension.coho} satisfies 
\begin{align*}
\nu_{0} & =\nu\left(X_{1}\right)-\left(\textup{number of free points}\right)\\
 & =\left\{ \begin{array}{ll}
\frac{n}{2}-1 & \textup{if \ensuremath{n} is even}\\
\frac{n-1}{2} & \textup{else}
\end{array}\right.-\left(\left\lceil \frac{n}{2}\right\rceil -2\right)=1.
\end{align*}
Now, following Propostion \ref{dimension.coho}, the dimension of
$\moduli$ is equal to 
\[
{\displaystyle \left\{ \begin{array}{rlc}
\frac{1}{2}\left(\frac{n}{2}-2\right)\left(\frac{n}{2}-3\right)+\frac{1}{2}\left(\frac{n}{2}-1\right)\left(\frac{n}{2}-2\right)+n-3 & =\frac{\left(n-2\right)^{2}}{4} & \textup{ if \ensuremath{n} is even}\\
\left(\frac{n-1}{2}-1\right)\left(\frac{n-1}{2}-2\right)+n-3 & =\frac{\left(n-1\right)\left(n-3\right)}{4} & \textup{ if \ensuremath{n} is odd}
\end{array}\right.}
\]
which coincides with the results in \cite{Granger}.

\bibliographystyle{plain}
\bibliography{Bibliographie}

\begin{thebibliography}{10}

\bibitem{MR922433}
J.~Brian{\c{c}}on, M.~Granger, and Ph. Maisonobe.
\newblock Le nombre de modules du germe de courbe plane {$x^a+y^b=0$}.
\newblock {\em Math. Ann.}, 279(3):535--551, 1988.

\bibitem{semigroup}
E.~Carvalho and M.~E. Hernandes.
\newblock Standard bases for fractional ideals of the local ring of an
  algebroid curve.
\newblock {\em J. Algebra}, 551:342--361, 2020.

\bibitem{MR704017}
D.~Cerveau and J.-F. Mattei.
\newblock {\em Formes int\'egrables holomorphes singuli\`eres}, volume~97 of
  {\em Ast\'erisque}.
\newblock Soci\'et\'e Math\'ematique de France, Paris, 1982.
\newblock With an English summary.

\bibitem{MR909850}
F.~Delgado de~la Mata.
\newblock The semigroup of values of a curve singularity with several branches.
\newblock {\em Manuscripta Math.}, 59(3):347--374, 1987.

\bibitem{Delorme1978}
C.~Delorme.
\newblock Sur les modules des singularit\'es des courbes planes.
\newblock {\em Bulletin de la Soci\'et\'e Math\'ematique de France},
  106:417--446, 1978.

\bibitem{diff}
H.~Dulac.
\newblock Recherches sur les points singuliers des \'equations
  diff\'erentielles.
\newblock {\em J. Ecole Polytechnique}, (2):1--125, 1904.

\bibitem{MR0176983}
S.~Ebey.
\newblock The classification of singular points of algebraic curves.
\newblock {\em Trans. Amer. Math. Soc.}, 118:454--471, 1965.

\bibitem{YoyoBMS}
Y.~Genzmer.
\newblock Dimension of the {M}oduli {S}pace of a {G}erm of {C}urve in
  {$\Bbb{C}$}2.
\newblock {\em Int. Math. Res. Not. IMRN}, (5):3805--3859, 2022.

\bibitem{MR4082253}
Y.~Genzmer and M.~E. Hernandes.
\newblock On the {S}aito basis and the {T}jurina number for plane branches.
\newblock {\em Trans. Amer. Math. Soc.}, 373(5):3693--3707, 2020.

\bibitem{MR2808211}
Y.~Genzmer and E.~Paul.
\newblock Normal forms of foliations and curves defined by a function with a
  generic tangent cone.
\newblock {\em Mosc. Math. J.}, 11(1):41--72, 181, 2011.

\bibitem{PaulGen}
Y.~Genzmer and E.~Paul.
\newblock Moduli spaces for topologically quasi-homogeneous functions.
\newblock {\em Journal of Singularities}, (14):3--33, 2016.

\bibitem{Gomez}
X.~G\'{o}mez-Mont.
\newblock The transverse dynamics of a holomorphic flow.
\newblock {\em Ann. of Math. (2)}, 127(1):49--92, 1988.

\bibitem{Granger}
J.-M. Granger.
\newblock Sur un espace de modules de germe de courbe plane.
\newblock {\em Bull. Sci. Math. (2)}, 103(1):3--16, 1979.

\bibitem{MR2509045}
A.~Hefez and M.~E. Hernandes.
\newblock Analytic classification of plane branches up to multiplicity 4.
\newblock {\em J. Symbolic Comput.}, 44(6):626--634, 2009.

\bibitem{MR2781209}
A.~Hefez and M.~E. Hernandes.
\newblock The analytic classification of plane branches.
\newblock {\em Bull. Lond. Math. Soc.}, 43(2):289--298, 2011.

\bibitem{MR2996882}
A.~Hefez and M.~E. Hernandes.
\newblock Algorithms for the implementation of the analytic classification of
  plane branches.
\newblock {\em J. Symbolic Comput.}, 50:308--313, 2013.

\bibitem{Hernandes.semiring}
M.~E. Hernandes and E.~de~Carvalho.
\newblock The value semiring of an algebroid curve.
\newblock {\em Communications in Algebra}, 48(8):3275--3284, 2020.

\bibitem{hertlingfor}
C.~Hertling.
\newblock Formules pour la multiplicit\'e et le nombre de {M}ilnor d'un
  feuilletage sur {$({\Bbb C}\sp 2,0)$}.
\newblock {\em Ann. Fac. Sci. Toulouse Math. (6)}, 9(4):655--670, 2000.

\bibitem{MR1101844}
O.~A. Laudal, B.~Martin, and G.~Pfister.
\newblock Moduli of plane curve singularities with {${\bf C}^*$}-action.
\newblock In {\em Singularities ({W}arsaw, 1985)}, volume~20 of {\em Banach
  Center Publ.}, pages 255--278. PWN, Warsaw, 1988.

\bibitem{MatQuasi}
J.-F. Mattei.
\newblock Quasi-homog\'en\'eit\'e et \'equir\'eductibilit\'e de feuilletages
  holomorphes en dimension deux.
\newblock {\em Ast\'erisque}, (261):xix, 253--276, 2000.
\newblock G\'eom\'etrie complexe et syst\`emes dynamiques (Orsay, 1995).

\bibitem{MM}
J.-F. Mattei and R.~Moussu.
\newblock Holonomie et int\'egrales premi\`eres.
\newblock {\em Ann. Sci. \'Ecole Norm. Sup. (4)}, 13(4):469--523, 1980.

\bibitem{MR586450}
K.~Saito.
\newblock Theory of logarithmic differential forms and logarithmic vector
  fields.
\newblock {\em J. Fac. Sci. Univ. Tokyo Sect. IA Math.}, 27(2):265--291, 1980.

\bibitem{SiuThm}
Y.T. Siu.
\newblock Every {S}tein subvariety admits a {S}tein neighborhood.
\newblock {\em Invent. Math.}, 38(1):89--100, 1976/77.

\bibitem{Waldi}
R.~Waldi.
\newblock Wertehalbgruppe und singularitt einer ebenen algebraischen kurve.
\newblock {\em Dissertation}, Regenburg, 1972.

\bibitem{zariski}
O.~Zariski.
\newblock {\em Le probl\`eme des modules pour les branches planes}.
\newblock Hermann, Paris, second edition, 1986.
\newblock Course given at the Centre de Math\'ematiques de l'\'Ecole
  Polytechnique, Paris, October--November 1973, With an appendix by Bernard
  Teissier.

\end{thebibliography}

Genzmer, Yohann \\
yohann.genzmer@math.univ-toulouse.fr
\end{document}